\documentclass[11pt,a4paper,headinclude,footinclude,fleqn,reqno]{amsart}                 
\usepackage[T1]{fontenc}                   
\usepackage[utf8]{inputenc}                 
\usepackage[english]{babel}       
\usepackage{graphicx}                      %
\usepackage[font=small]{quoting}            %
\usepackage{caption}  
\usepackage{amsmath}
\usepackage{amsthm}
\usepackage{amssymb}            
\usepackage[top=.8in,bottom=1.2in,left=1.5in,right=1.5in]{geometry}
\usepackage{verbatim}
\usepackage{color}
\usepackage{picture}
\usepackage{graphicx}
\usepackage{graphics}
\usepackage{comment}
\usepackage{hyperref}
\hypersetup{colorlinks,linkcolor={blue},citecolor={blue},urlcolor={red}}  
\newtheorem{theorem}{\bf Theorem}[section]
\newtheorem{definition}[theorem]{\bf Definition}
\newtheorem{lemma}[theorem]{\bf Lemma}
\newtheorem{prop}[theorem]{\bf Proposition}

\newtheorem{remark}[theorem]{\bf Remark}

\def\no{\noindent}

\def\Om{\Omega}
\def\pa{\partial}

\def\ov{\overline}

\makeatletter

\newcommand{\Rmnum}[1]{\expandafter\@slowromancap\romannumeral #1@}
\makeatother \numberwithin{equation}{section}

\begin{document}
	
	\begin{center}{\LARGE \bf Rigidity of conformal minimal
			immersions of constant curvature from $S^2$ to $Q_4$}
	\end{center}

	\begin{center}
		Xiaoxiang Jiao,
		\footnote{
			X.X. Jiao
			
			School of Mathematical Sciences, University of Chinese Academy of
			Sciences, Beijing 100049, P. R. China
			
			e-mail: xxjiao@ucas.ac.cn}
		Mingyan Li
		\footnote{
			M.Y. Li (Corresponding author)
			
			School of Mathematics and Statistics, Zhengzhou University, Zhengzhou 450001, P. R. China
			
			e-mail: limyan@zzu.edu.cn
		}
		and 
		Hong Li
		\footnote{
			H. Li 
			
			School of Mathematical Sciences, University of Chinese Academy of
			Sciences, Beijing 100049, P. R. China
			
			e-mail: lihong16@mails.ucas.ac.cn}
	\end{center}

	\bigskip
	
	\no
	{\bf Abstract.}
	Geometry of  conformal minimal two-spheres immersed in $G(2,6;\mathbb{R})$  is studied in this paper by harmonic maps.
	We construct a non-homogeneous constant curved minimal two-sphere in $G(2,6;\mathbb{R})$, and give a classification theorem of linearly full  conformal minimal
	immersions of constant curvature from $S^2$ to $G(2,6;\mathbb{R})$, or equivalently, a complex hyperquadric $Q_{4}$, which illustrates minimal two-spheres  of constant curvature in $Q_{4}$ are in general not congruent.\\
	
	\no
	
	\no
	
	\no
	Project supported by the NSFC (Grant No. 11871450).
	
	\bigskip

	\bigskip
	
	\section{Introduction}\label{sec1}
	It is a long history of studying conformal minimal two-spheres with constant curvature
	in various Riemannian spaces (see \cite{2, 3, 5, 12}). In 1988 Bolton et al
	\cite{2} studied properties about conformal minimal two-spheres in a complex projection space $\mathbb{C}P^n$ and proved that any linearly full conformal minimal immersion
	of constant curvature
	from $S^2$ to $\mathbb{C}P^n$ belongs to the Veronese sequence, up to a rigid motion.  It is well known that, this rigidity fails for conformal minimal two-spheres of constant curvature immersed in general Riemannian symmetric spaces,
	for example, complex Grassmannian $G(k,n;\mathbb{C})$, complex
	hyperquadric $Q_n$ and quaternionic projective space $HP^n$ and so
	on. Recently, we got a classification theorem of
	linearly full totally unramified conformal minimal immersions of
	constant curvature from $S^2$ to $Q_3$ (\cite{13}), which showed that all such immersions can be presented by Veronese curves in $\mathbb{C}P^4$ (\cite{13}, Theorem 4.9). For general linearly full totally unramified conformal minimal two-spheres immersed in complex hyperquadric  $Q_n$, we obtained a classification theorem under some conditions (\cite{10}, Theorem 4.6).

	As is well known, complex
	hyperquadric $Q_{n-2}$ may be identified with $G(2,n;\mathbb{R})$, which is considered as a totally
	geodesic submanifold in complex Grassmann manifold
	$G(2,n;\mathbb{C})$ (for detailed
	descriptions see the Preliminaries below). In 1986 Burstall and
	Wood \cite{4} gave the explicit construction of all two-spheres in $G(2,n;\mathbb{C})$, they pointed out that, any harmonic map from $S^2$ to $G(2,n;\mathbb{C})$ can be obtained from a holomorphic map, a Frenet pair or a mixed pair. For the special case $G(2,n;\mathbb{R})$, Bahy-El-Dien and
	Wood \cite{1} gave their explicit construction in 1989.

	The purpose of this paper is to apply the method of harmonic maps they gave and to derive a classification of conformal minimal immersions
	of constant curvature from $S^2$ to $G(2,6;\mathbb{R})$.  By doing this, we hope that some insight can be gained on geometry for general cases , i.e.,  for any positive integer $n$.

	It is well known that constant curved minimal two-spheres in $S^n(1)$ and $\mathbb{C}P^n$ are homogeneous, they also determined the values distribution of the constant curvature completely.  Papers \cite{10, 13} proved that constant curved minimal two-spheres in  $Q_2$ and   $Q_3$ are also homogeneous, a natural question is the following:
	
	\textbf{Problem.} Does the minimal two-spheres with constant curvature in  $Q_{n}$ must be homogeneous?
	
	In this paper, we give a negative answer to the problem stated above by constructing a non-homogeneous constant curved minimal two-sphere in $Q_{4}$:
	$$f^{(3)}_0=[(1+z^3, \  \sqrt{-1}(1-z^3),   \  \sqrt{3}z-\frac{z^2}{\sqrt{3}},   \  \sqrt{-1}(\sqrt{3}z+\frac{z^2}{\sqrt{3}}),   \ 
	\frac{\sqrt{8}}{\sqrt{3}}z^2,  \ 
	\frac{\sqrt{-8}}{\sqrt{3}}z^2)^T].$$
	Here $f^{(3)}_0$ is of constant Gauss curvature $\frac{2}{3}$. It  is the first curve that is non-homogeneous under the assumption of constant Gauss curvature, which we cannot find yet in any literature and made much effort for it. It plays a key role in our later work.

	Our paper is organized as follows.  In Section 2,  we
	identify $Q_{n-2}$ and $G(2,n;\mathbb{R})$,  state some
	fundamental results concerning $G(k,n;\mathbb{C})$ from the view of
	harmonic sequences.  In Section 3, we  introduce the definition of degree of a smooth map from a compact Riemann surface to  $G(k,n;\mathbb{C})$, and then show some brief descriptions of
	Veronese sequence and the rigidity theorem in $\mathbb{C}P^n$. In
	Section 4,  we  present some properties of the harmonic sequences generated by  reducible harmonic maps from $S^2$ to $G(2,6;\mathbb{R})$, and obtain the explicit characteristics of the corresponding harmonic maps in $G(2,6;\mathbb{R})$.  Moreover we classify all  reducible harmonic maps of  $S^2$ in $G(2,6;\mathbb{R})$ under the assumption that they have constant curvature (see Proposition 4.3).  In  Section 5, using Burstall, Bahy-El-Dien and Wood's results \cite{1, 4} , we discuss geometric properties of irreducible harmonic maps of two-spheres  in $G(2,6;\mathbb{R})$ with constant curvature and give a classification theorem of linearly full totally unramified  conformal minimal
	immersions of constant curvature from $S^2$ to $G(2,6;\mathbb{R})$,
	or equivalently, a complex hyperquadric $Q_{4}$ (see Theorem 5.7).

	\section{Minimal immersions and harmonic sequences in $G(k,n;\mathbb{C})$}\label{sec2}

	For $0 < k < n$,  we consider complex Grassmann manifold  $G(k,n;\mathbb{C})$ as the
	set of Hermitian orthogonal projections from $\mathbb{C}^n$ onto a
	$k$-dimensional subspace in $\mathbb{C}^n$. Here $\mathbb{C}^n$ is endowed with the Hermitian inner product $\langle \cdot, \cdot\rangle$ defined by
	$$\langle x, y\rangle=x_1\overline{y}_1+\cdots+x_n\overline{y}_n,$$ where $x=(x_1,\ldots,x_n)^{T}$ and $y=(y_1,\ldots,y_n)^{T}$ are two elements of $\mathbb{C}^n$.

	Let $G(k,n;\mathbb{R})$ denote the
	Grassmannian of all real $k$-dimensional subspaces of $\mathbb{R}^n$
	and $$\sigma: G(k,n;\mathbb{C}) \rightarrow G(k,n;\mathbb{C})$$
	denote the complex conjugation of $G(k,n;\mathbb{C})$. It is easy to
	see that $\sigma$ is an isometry with the standard Riemannian metric
	of $G(k,n;\mathbb{C})$, its fixed point set is $G(k,n;\mathbb{R})$.
	thus $G(k,n;\mathbb{R})$ lies totally geodesically  in
	$G(k,n;\mathbb{C})$.
	
	Map $$Q_{n-2} \rightarrow G(2,n;\mathbb{R})$$ by $$q \mapsto
	\frac{\sqrt{-1}}{2} \mathrm{Z} \wedge \overline{\mathrm{Z}},$$ where $q\in Q_{n-2}$
	and $\mathrm{Z}$ is a homogeneous coordinate vector of $q$. It is clear that
	the map  is
	one-to-one and onto, and it is an isometry. Thus we can identify
	$Q_{n-2}$ and $G(2,n;\mathbb{R})$ (for more details see \cite{15}).

	Next, we simply introduce general expressions of some geometric
	quantities about conformal minimal immersions from $S^2$ to complex
	Grassmannian manifold $G(k,n;\mathbb{C})$.

	Let $U(n)$ be the unitary group, $M$ be a simply connected domain in the unit sphere $S^2$ and $(z, \overline{z})$ be a complex coordinate on $M$. We take the metric $ds_M^2=dzd\overline{z}$ on $M$. Denote
	$$A_z=\frac{1}{2}s^{-1}\partial s, \quad A_{\overline{z}}=\frac{1}{2}s^{-1}\overline{\partial} s,$$
	where $s: M\rightarrow U(n)$ is a smooth map, $\partial = \frac {\partial}{\partial z}, \ \overline{\partial}
	= \frac {\partial}{\partial \overline{z}}$.
	
	Then $s$ is a harmonic map if and only if it satisfies the following equation (cf. \cite{14}):
	$$\overline{\partial} A_z=[A_z, A_{\overline{z}}]. $$
	Suppose that $s:S^2 \rightarrow U(n)$ is an isometric immersion, then $s$ is conformal and minimal if it is harmonic. Let $\omega$ be the Maurer-Cartan form on $U(n)$, and let $ds_{U(n)}^2=\frac{1}{8}tr\omega\omega^{*}$ be the metric on $U(n)$. Then the metric induced by $s$ on $S^2$ is locally given by
	$$ds^2=-tr A_zA_{\overline{z}}dzd\overline{z}.$$

	We consider the complex Grassmann manifold  $G(k,n;\mathbb{C})$ as
	the set of Hermitian orthogonal projection from $\mathbb{C}^n$ onto
	a $k$-dimensional subspace in $\mathbb{C}^n$. Then a map $\phi:
	M\rightarrow G(k,n;\mathbb{C})$ is a Hermitian orthogonal projection
	onto a $k$-dimensional subbundle $\underline{\phi}$ of the
	trivial bundle $\underline{\mathbb{C}}^n = M \times \mathbb{C}^n$
	given by setting the fibre $\underline{\phi}_x = \phi (x)$ for
	all $x\in M$.  $\underline{\phi}$ is called (a)
	\emph{harmonic ((sub-) bundle) } whenever $\phi$ is a harmonic
	map. Here $s=\phi-\phi^{\bot}$ is a map from $S^2$ into $U(n)$. It is well known that $\phi$ is harmonic if and only if $s$ is harmonic.

	For a conformal minimal immersion $\phi: S^2 \rightarrow
	G(k,n;\mathbb{C})$,two harmonic sequences are derived as follows:
	\begin{equation}\underline{\phi}=\underline{\phi}_{0}\stackrel{\partial{'}}
		{\longrightarrow }\underline{\phi}_1\stackrel{\partial{'}}
		{\longrightarrow}\cdots\stackrel{\partial{'}}
		{\longrightarrow}\underline{\phi}_ {i}\stackrel{\partial{'}}
		{\longrightarrow}\cdots, \label{eq:2.1}
	\end{equation}
	\begin{equation}\underline{\phi}=\underline{\phi}_{0}\stackrel{\partial{''}}
		{\longrightarrow}\underline{\phi}_{-1}\stackrel{\partial{''}}
		{\longrightarrow}\cdots\stackrel{\partial{''}} {\longrightarrow}
		\underline{\phi}_{-i}\stackrel{\partial{''}}{\longrightarrow}\cdots,
		\label{eq:2.2}
	\end{equation} where $\underline{\phi}_{i} = \partial ^\prime
	\underline{\phi} _{i-1}$ and $\underline{\phi}_{-i}=
	\partial ^{ \prime \prime} \underline{\phi} _{-i+1} $ are Hermitian
	orthogonal projections from $S^2 \times \mathbb{C} ^n$ onto
	${\underline{Im}}\left(\phi^{\perp}_ {i-1}\partial \phi_{i-1}
	\right)$ and ${\underline{Im}}\left(\phi^{\perp}_
	{-i+1}\overline{\partial} \phi_ {-i+1}\right)$ respectively, in the following we also denote them by $\partial^{(i)}\underline{\phi}$ and $\partial^{(-i)}\underline{\phi}$ respectively,
	$i=1,2,\ldots$.

	Now recall (\cite{4}, \S3A) that a harmonic map $\phi: S^2
	\rightarrow G(k,n;\mathbb{C})$ in \eqref{eq:2.1} (resp.
	\eqref{eq:2.2}) is said to be \emph{$\partial^{'}$-irreducible}
	(resp. \emph{$\partial^{''}$-irreducible}) if rank
	$\underline{\phi} =$ rank $\underline{\phi}_1$ (resp. rank
	$\underline{\phi} =$ rank $\underline{\phi}_{-1}$) and
	\emph{$\partial^{'}$-reducible} (resp.
	\emph{$\partial^{''}$-reducible}) otherwise. In particular, if
	$\phi$ is a harmonic map from $S^2$ to $G(k,n;\mathbb{R})$, then
	$\phi$ is $\partial^{'}$-irreducible (resp.
	$\partial^{'}$-reducible) if and only if $\phi$ is
	$\partial^{''}$-irreducible (resp. $\partial^{''}$-reducible). In
	this case we simply say that $\phi$ is irreducible (resp.
	reducible).

	As in \cite{7} call a harmonic map $\phi: S^2 \rightarrow
	G(k,n;\mathbb{C})$ \emph{(strongly) isotropic} if $\phi_{i} \bot
	\phi,   \  \forall i\in \mathbb{Z}, \ i \neq 0$.
	
	For an arbitrary harmonic map $\phi: S^2 \rightarrow
	G(k,n;\mathbb{C})$, define its \emph{isotropy order} (cf. \cite{4})
	to be the greatest integer $r$ such that $\phi_{i} \bot \phi$
	for all $i$ with $1\leq i \leq r$;  if $\underline{\phi}$ is
	isotropic, set $r= \infty$.

	\begin{definition}
		\emph{Let $\phi: S^2 \rightarrow G(k,n;\mathbb{C})$ be a map.
			$\phi$ is called}  linearly full \emph{if $\underline{\phi}$ can not
			be contained in any proper trivial subbundle $S^2 \times
			\mathbb{C}^m$ of $S^2\times { \mathbb{C}}^n$ ($m< n$)}.
	\end{definition}

	In this paper, we always assume that $\phi$ is linearly full.

	Suppose that $\phi: S^2 \rightarrow G(k,n;\mathbb{C})$ is a
	linearly full harmonic map and it belongs to the following harmonic
	sequence
	\begin{equation}\underline{\phi}_0\stackrel{\partial{'}}
		{\longrightarrow}\cdots \stackrel{\partial{'}} {\longrightarrow}
		\underline{\phi}= \underline{\phi}_{i} \stackrel {\partial{'}}
		{\longrightarrow} \underline{\phi}_{i+1}\stackrel{\partial{'}}
		{\longrightarrow}\cdots \stackrel{\partial{'}} {\longrightarrow}
		\underline{\phi}_{i_0}\stackrel{\partial{'}} {\longrightarrow}
		0\label{eq:2.3}
	\end{equation} for some $i = 0, \ldots, i_0$. We choose the local unit
	orthogonal frame $e^{(i)}_1, e^{(i)}_2, \ldots, e^{(i)}_{k_{i}}$
	such that they locally span subbundle $\underline{\phi}_{i}$ of
	$S^2 \times \mathbb{C}^n$, where $k_i = $ rank
	$\underline{\phi}_{i}$.

	Let $W_{i}=\left(e^{(i)}_1, e^{(i)}_2, \ldots,
	e^{(i)}_{k_{i}}\right)$ be an $\left(n\times k_{i}\right)$-matrix. Then
	we have
	$$\phi_{i}=W_{i}W^*_{i},$$
	\begin{equation}
		W^*_{i}W_{i} = I_{k_{i}\times k_{i}}, \quad W^*_{i}W_{i+1} =0,\quad
		W^*_{i} W_{i-1}=0.\label{eq:2.4}\end{equation} 
	By \eqref{eq:2.4}, a
	straightforward computation shows that
	\begin{equation} \left\{\begin{array}{l} \pa W_{i} = W_{i+1} \Om_{i} +
			W_{i} \Psi_{i}, \\
			\ov{\pa}W_{i} =-W_{i-1} \Om^*_ {i-1}- W_{i}\Psi^*_{i},
		\end{array}\right.  \label{eq:2.5}\end{equation}
	where $\Omega_{i}$ is a $\left(k_{i+1} \times k_{i}\right)$-matrix,  $\Psi_{i}$ is a $\left(k_{i}\times k_{i}\right)$-matrix for
	$i=0, 1, 2, \ldots, i_0$ and $\Omega_{i_0}=0$. It is very evident that integrability conditions for  \eqref{eq:2.5} are
	$$\overline{\partial}\Omega_i=\Psi_{i+1}^{*}\Omega_i-\Omega_i\Psi_{i}^{*},$$
	$$\overline{\partial}\Psi_i+\partial\Psi_{i}^{*}=\Omega_i^*\Omega_i+\Psi_i^*\Psi_i-\Omega_{i-1}\Omega_{i-1}^{*}-\Psi_i\Psi_i^*.$$

	Now we assume that $\phi_{i}$ is $\partial^{'}$-irreducible, then $|\det \Omega_{i}|^2dz^{k_i}d\overline{z}^{k_i}$ is a well-defined invariant on $S^2$ and has only isolated zeros.

	Set $L_{i} =\textrm{tr} (\Omega_{i} \Omega^*_{i})$, the metric induced by $\phi_{i}$ is
	given in the form  \begin{equation}ds^2 _{i} =
		(L_{i-1}+L_{i})dzd\overline{z}. \label{eq:2.6}\end{equation}
	The  Gauss curvature $K$ and second fundamental form $B$  of $\phi_{i}$
	are given by \begin{equation} \left\{
		\begin{array}{l}
			K = -\frac{2}{L_{i-1}+L_{i}}\partial\overline{\partial} \log {(L_{i-1}+L_{i})},
			\\ \|B\|^2 = 4\textrm{tr}PP^{\ast},
		\end{array}\right.  \label{eq:2.7}\end{equation}
	where $P=\partial\left(\frac{A_z}{\lambda^2}\right), \
	P^{\ast} =- \overline{\partial}
	\left(\frac{A_{\overline{z}}}{\lambda^2}\right)$ with
	$\lambda^2=L_{i-1}+L_{i}$ (cf. \cite{11}).

	In the following, we give a definition of the unramified harmonic
	map as follows.

	\begin{definition}[{\cite{9}}] 
		\emph{If $\det (\Omega _i \Omega^*_i) dz^{k_{i +  1}} d\overline{z}
			^{k_{i+ 1}} \neq 0$ everywhere on $S^2$ in \eqref{eq:2.3} for some
			$i$, we say that $\phi_i: S^2 \rightarrow G(k_i,n; \mathbb{C})$
			is} unramified. \emph{If $\det (\Omega _i \Omega^*_i) dz^{k_{i + 1}}
			d\overline{z} ^{k_{i+ 1}} \neq 0$ everywhere on $S^2$ in
			\eqref{eq:2.1} (resp. \eqref{eq:2.2}) for each $i = 0, 1, 2,
			\ldots$, we say that the harmonic sequence \eqref{eq:2.1} (resp.
			\eqref{eq:2.2}) is} totally unramified. \emph{If \eqref{eq:2.1} and
			\eqref{eq:2.2} are both totally unramified, we say that $\phi$
			is} totally unramified.
	\end{definition}

	Especially, let $\psi:S^2 \rightarrow \mathbb{C}P^n$ be a linearly full
	conformal minimal immersion,  then the following  harmonic sequence in $\mathbb{C}P^n$  is uniquely determined by $\psi$
	\begin{equation}0\stackrel{\partial^\prime}{\longrightarrow}
		\underline{\psi} _0^{(n)}
		\stackrel{\partial^\prime}{\longrightarrow}\cdots
		\stackrel{\partial^\prime}{\longrightarrow} \underline{\psi}
		=\underline{\psi}_i^{(n)} \stackrel{\partial^\prime}{\longrightarrow}
		\cdots \stackrel{\partial^\prime}{\longrightarrow} \underline{\psi}
		_n^{(n)} \stackrel{\partial^\prime}{\longrightarrow} 0
		\label{eq:2.8}\end{equation} for some $i =0, 1, \ldots,n$. In the following we also denote \eqref{eq:2.8} by $\psi _0^{(n)}, ..., \psi _n^{(n)}: S^2 \rightarrow \mathbb{C}P^n$. 
	
	Define a sequence $f_0^{(n)}, \ldots, f_n^{(n)}$ of local sections of
	$\underline{\psi}_0^{(n)}, \ldots , \underline{\psi}_n^{(n)}$ inductively such
	that $f_0^{(n)}$ is a nowhere zero local section of $\underline{\psi}_0^{(n)}$
	(without loss of generality, assume that $\overline{\partial} f_0^{(n)}
	\equiv 0$ ) and $f_{i+1}^{(n)} = \psi^{(n)\perp}_{i}(
	\partial f_{i}^{(n)}) $ for $i=0, \ldots,n-1$. Then we have
	some formulae as follows: \begin{equation}\partial f_i^{(n)} = f_{i+1}^{(n)} +
		\frac{\langle\partial f_i^{(n)},f_i^{(n)}\rangle}{|f_i^{(n)}|^2}f_i^{(n)}, \ i= 0,\ldots, n,
		\label{eq:2.9}\end{equation} 
	\begin{equation} \overline{\partial}
		f_i^{(n)}= - \frac{|f_i^{(n)}|^2}{|f_{i-1}^{(n)}|^2} f_{i-1}^{(n)}, \ i = 1, \ldots, n.
		\label{eq:2.10}\end{equation}

	\section{Degrees of harmonic maps in $G(k,n;\mathbb{C})$}\label{sec3}
	
	In this section we state the definition of degree of a smooth map $\phi$ from a compact Riemann surface $M$ into $G(k,n;\mathbb{C})$ as follows.

	\begin{definition}[{\cite{4}}] 
		\emph{The} degree \emph{of}  $\phi$, \emph{denoted by deg(}$\phi$\emph{) is the degree of the induced map} $\phi^*: H^2(G(k,n;\mathbb{C}),\mathbb{Z})\cong \mathbb{Z} \longrightarrow H^2(M,\mathbb{Z})\cong \mathbb{Z}$ \emph{on second cohomology.}\end{definition}
	
	\begin{definition}[{\cite{6}}] 
		\emph{Let } $\phi:S^2 \rightarrow G(k,n;\mathbb{C})$ \emph{be a harmonic map.} $\phi$ \emph{is called a} pseudo-holomorphic curve \emph{if it is obtained by some holomorphic curve via $\partial^{'}$}   \emph{in}  \eqref{eq:2.3}. \end{definition}

	Now let $\phi$ be a linearly  full pseudo-holomorphic curve in $G(k,n;\mathbb{C})$ with the harmonic sequence
	\begin{equation}0\stackrel{\partial^\prime}{\longrightarrow}
		\underline{\phi} _0
		\stackrel{\partial^\prime}{\longrightarrow}\cdots
		\stackrel{\partial^\prime}{\longrightarrow} \underline{\phi}
		=\underline{\phi}_i\stackrel{\partial^\prime}{\longrightarrow}
		\cdots \stackrel{\partial^\prime}{\longrightarrow} \underline{\phi}
		_n\stackrel{\partial^\prime}{\longrightarrow} 0,
		\label{eq:3.1}\end{equation} 
	let $\phi^{(i)}=\phi_0\oplus\phi_1\oplus\cdot\cdot\cdot\oplus\phi_i$, where  $i=0, 1, 2,...$. Then $\phi^{(i)}$ is holomorphic, and $\partial^{'}\phi^{(i)}=\phi_{i+1}$.

	Let $\Im:G(k,n;\mathbb{C}) \rightarrow \mathbb{C}P^N$ be the Pl\"{u}cker embedding, and let $F^{(i)}$ be a nowhere zero holomorphic section of $\underline{Im}(\Im\circ \phi^{(i)})$,  it follows that
	$$\partial\overline{\partial}\log|F^{(i)}|^2=L_i.$$
	Denote the degree of $\phi^{(i)}$ by $\delta_{i}$. Then  
	
	\begin{equation} \delta_{i}= \frac{1}{2\pi\sqrt{-1}
		}\int_{S^2}\partial\overline{\partial}\log|F^{(i)}|^2d{\overline{z}} \wedge dz = \frac{1}{2\pi\sqrt{-1}
		}\int_{S^2}L_{i}d{\overline{z}} \wedge dz.
		\label{eq:3.2}\end{equation} 
	We state the following results.
	
	\begin{lemma}[\cite{9}] 
		Let $\phi=\phi_{i}:S^2 \rightarrow G(k,n;\mathbb{C})$ be a linearly full pseudo-holomorphic curve in \eqref{eq:3.1}. Then
		\\ (1) $deg(\phi)= \delta_{i}-\delta_{i-1}$;
		\\ (2) Suppose $\phi$ is $\partial^{'}$-irreducible, and $|\det \Omega_{i}|^2dz^{k_{i}} d\overline{z} ^{k_{i}}$ is a
		well-defined invariant and has no zeros on $S^2$, then $\delta_{i-1}-2\delta_{i}+\delta_{i+1}=-2k_i.$
	\end{lemma}

	Especially, for the harmonic sequence \eqref{eq:2.8}, let $r(\partial^{'})=$ sum of the indices of the singularities of $\partial^{'}$, which is called the \emph{ramification index} of $\partial^{'}$ by Bolton et al(cf. \cite{2}). Note that if $r(\partial^{'})=0$ in \eqref{eq:2.8} for all $\partial^{'}$, the harmonic sequence \eqref{eq:2.8} is defined \emph{totally unramified} in \cite{2}.

	For the harmonic sequence $\psi _0^{(n)}, ..., \psi _n^{(n)}: S^2 \rightarrow \mathbb{C}P^n$,  let $l_i^{(n)}= \frac{|f^{(n)}_{i+1}|^2}{|f^{(n)}_i|^2}$ and
	$\delta_i^{(n)} = \frac{1}{2\pi\sqrt{-1}}\int_{S^2}l^{(n)}_i d\overline{z}
	\wedge dz$,  $i =0, \ldots, n-1, \ l^{(n)}_{-1} = l^{(n)}_n =0$. It is easy to
	check that they are in accordance with $L_i$ and $\delta_i$
	respectively in the case $k=1$.  Bolton et al showed (\cite{2})
	$$\delta_ i^{(n)} = (i+1)(n-i)+\frac{n-i}{n+1}\sum_{k=0}^{i-1}(k+1)r(\partial_k)+\frac{i+1}{n+1}\sum_{k=i}^{n-1}(n-k)r(\partial_k).$$
	In particular for  a
	totally unramified harmonic sequence $\psi _0^{(n)}, ..., \psi _n^{(n)}: S^2 \rightarrow \mathbb{C}P^n$ (i.e. $\psi_i^{(n)}$ is unramified,
	$i=0,\ldots, n$), Bolton et al proved (cf. \cite{2})
	\begin{equation}\delta_ i^{(n)} = (i+1)(n-i).\label{eq:3.3}\end{equation}

	In the final of this section we state the rigidity theorem of conformal minimal immersions of $S^2$ into $\mathbb{C}P^n$ with constant curvature as follows.
	Consider the \emph{Veronese sequence}
	$$0{\longrightarrow }\underline{V}^{(n)}_0\stackrel{\partial{'}}
	{\longrightarrow}\underline{V}^{(n)}_1\stackrel{\partial{'}}
	{\longrightarrow}\cdots\stackrel{\partial{'}}
	{\longrightarrow}\underline{V}^{(n)}_n\stackrel{\partial{'}}
	{\longrightarrow}0.$$ For each $i=0,\ldots, n$,  $V^{(n)}_i: S^2\rightarrow \mathbb{C}P^n$ is given by $V^{(n)}_i = (v_{i,0}, \ldots, v_{i,n})^{T}$,  where, for
	$z\in S^2$ and $j=0,\ldots, n$,  $$v_{i,j}(z) =
	\frac{i!}{(1 +
		z\overline{z})^i}\sqrt{\binom{n}{j}}z^{j-i}\sum_{k}(-1)^k
	\binom{j}{i-k}\binom{n-j}{k}(z\overline{z})^k.$$
	Here map $\underline{V}^{(n)}_i: S^2\rightarrow \mathbb{C}P^n$ is a conformal minimal immersion with   induced metric $ds^2_i= \frac{n
		+ 2i(n - i)}{(1+z\overline{z})^2}dzd\overline{z}$ and  constant curvature $K_i= \frac{4}{n +
		2i(n-i)}$.

	By Calabi's rigidity theorem, Bolton et al proved the following
	rigidity result (cf.\cite{2}).
	\begin{lemma}[{\cite{2}}]
		Let $\psi : S^2 \rightarrow \mathbb{C}P^n$ be a linearly full
		conformal minimal immersion of constant curvature. Then, up to a
		holomorphic isometry of $\mathbb{C}P^n$, $\psi$ is a member of the Veronese sequence.
	\end{lemma}

	\section{Reducible  harmonic maps of constant curvature}
	
	In the following, we regard harmonic maps from $S^2$ to $G(2,6;\mathbb{R})$ as conformal minimal immersions of 
	$S^2$ in $G(2,6;\mathbb{R})$. Then we analyze harmonic maps of constant curvature from $S^2$ to $G(2,6;\mathbb{R})$ by
	reducible case and irreducible case and divide them into two sections.

	In this section we first  discuss the reducible ones.  Let  $\phi:S^2 \rightarrow G(2,6;\mathbb{R})$ be a linearly full reducible harmonic map with constant curvature, 
	it follows from (\cite{1}, Proposition 2.12),  to finish the characterize of $\phi$, we distinguish two cases:
	\\ (1) $\phi$ is  a real mixed pair with finite isotropy order, whereas
	\\ (2) $\phi$ is (strongly) isotropic.

	We  first briefly discuss the case that $\phi$ has finite isotropy order, suppose $\phi: S^2 \rightarrow G(2,6;\mathbb{R})$ is a linearly full
	reducible harmonic map with constant curvature and finite isotropy order $r$. It follows from (\cite{10}, Proposition 3.2) that $r=1$,  and then $\phi$ can be
	characterized by  harmonic maps from $S^2$ to $\mathbb{C}P^m (m \leq 5)$, in fact,  $$\underline{\phi} =
	\underline{\overline{f}}_0^{(m)} \oplus \underline{f}_0^{(m)},$$ where $f_0^{(m)}: S^2
	\rightarrow \mathbb{C}P^m$ is holomorphic.
	
	By using $\phi$, a harmonic
	sequence is derived as follows
	\begin{equation}0\stackrel{\partial{''}} {\longleftarrow}
		\overline{\underline{f}}_m^{(m)}\stackrel{\partial{''}} {\longleftarrow}
		\cdots\stackrel{\partial{''}} {\longleftarrow}
		\overline{\underline{f}}_1^{(m)}\stackrel{\partial{''}} {\longleftarrow}
		\underline{\phi}\stackrel {\partial{'}}
		{\longrightarrow}\underline{f}_1^{(m)} \stackrel {\partial{'}}
		{\longrightarrow}\cdots\stackrel {\partial{'}}
		{\longrightarrow}\underline{f}_m^{(m)} \stackrel{\partial{'}}
		{\longrightarrow}0,\label{eq:4.1}\end{equation}  
	where $0\stackrel{\partial{'}}
	{\longrightarrow} \underline{f}_0^{(m)}\stackrel{\partial{'}}
	{\longrightarrow}\underline{f}_1^{(m)}\stackrel{\partial{'}}
	{\longrightarrow} \cdots \stackrel {\partial{'}}
	{\longrightarrow}\underline{f}_m ^{(m)}\stackrel{\partial{'}}
	{\longrightarrow}0 $  
	is a linearly full harmonic sequence in
	$\mathbb{C}P^{m} \subset \mathbb{C}P^{5}$ satisfying
	\begin{equation}\left\{  \begin{array}{l} \langle {f}_0^{(m)},   \overline{f}_i^{(m)} \rangle=0   \  (i=0, 1),\\
			\langle {f}_0^{(m)}, \overline{f}_{2}^{(m)} \rangle\neq 0,
		\end{array}\right. \label{eq:4.2}\end{equation}
	and $\ {2}  \leq m\leq {5}$.   
	The induced metric of $\phi$
	is given by 
	\begin{equation}ds^2 = 2l_0^{(m)}dzd\overline{z},
		\label{eq:4.3}\end{equation} 
	where $l_0^{(m)}dzd\overline{z}$ is the
	induced metric of $f_0^{(m)}: S^2 \rightarrow \mathbb{C}P^m$.  Since $\underline{\phi}$ is of constant
	curvature, using \eqref{eq:4.3} we get that the curvature $K$ of
	$\phi$ satisfies $$K=\frac{2}{m}.$$
	
	By Lemma 3.4, up to a
	holomorphic isometry of $\mathbb{C}P^{5}$, ${f}_0^{(m)}$ is a Veronese
	surface. We can then choose a complex coordinate $z$ on
	$\mathbb{C}=S^2\backslash\{pt\}$ so that $ {f}_0^{(m)} = {U}{{V}}^{(m)}_0$, 
	where $U\in U(6)$ and ${V}^{(m)}_0$ has the standard expression
	given in Section 3 (adding zeros to ${V}^{(m)}_0$ such
	that ${V}^{(m)}_0 \in \mathbb{C}^6$).    
	Then \eqref{eq:4.2} becomes
	$$ \left\{
	\begin{array} {l} \langle UV^{(m)}_0,
	\overline{U}\overline{V}^{(m)}_i\rangle = 0 \  (i=0, 1), \\
	\langle UV^{(m)}_0,
	\overline{U}\overline{V}^{(m)}_2\rangle \neq 0,
	\end{array} \right. $$
	which is equivalent to \begin{equation}\left\{ \begin{array} {l}
			\textrm{tr} WV^{(m)}_0 V^{(m)T}_i= 0 \  (i=0, 1), \\\textrm{tr} WV^{(m)}_0
			V^{(m)T}_{2}\neq 0,
		\end{array} \right. \label{eq:4.4}\end{equation}
	where $W = U^{T}U$, it satisfies $W\in U(6)$ and $W^{T} = W$.

	For any integers $n, s$ with $n\geq 3, \ s\geq 0$, let $H^s_n$ denote the set of all holomorphic maps $f: S^2\rightarrow\mathbb{C}P^{n-1}$ satisfying condition 
	$$\left\{  \begin{array}{l} \langle \partial^{(i)}f,   \overline{f}\rangle=0   \  (0 \leq i \leq 2s+1),\\
	\langle \partial^{(2s+2)}f,   \overline{f}\rangle\neq 0.
	\end{array}\right. $$
	This together with \eqref{eq:4.2} implies that
	$f_0^{(m)} \in H^0_{m+1}$. To characterize $\phi$, here we state one of Bahy-El-Dien and Wood's results  as follows:

	\begin{lemma} [Special case of \cite{1}, Proposition 5.7]
		All  holomorphic maps $f_0^{(m)}: S^2\rightarrow\mathbb{C}P^{m}$ 
		satisfying    $f_0^{(m)} \in H^0_{m+1}$ may be constructed by the following three steps: 
		\\ (1) Choose $F_0(z): \mathbb{C}\rightarrow(\mathbb{C}\cup \{\infty\})^{m-1}$ polynomial with $\langle    F_0(z), \overline F_0(z)\rangle\neq 0$;  
		\\ (2) Let $H(z)$ be the unique rational function $\mathbb{C}\rightarrow(\mathbb{C}\cup \{\infty\})^{m-1}$ with $\frac{d H(z)}{dz}=F_0(z)$ for any $z \in \mathbb{C}$ and $H(0)=0$;
		\\ (3) Define $F_1(z): \mathbb{C}\rightarrow \mathbb{C}^{m+1}=\mathbb{C}^{m-1}\oplus\mathbb{C}\oplus\mathbb{C}$ by $F_1(z)=(2H(z), 1-\langle H(z),   \overline{H}(z)\rangle, \sqrt{-1}(1+\langle H(z),   \overline{H}(z)\rangle))$.  Then $F_1(z)$ is a rational function and so represents the holomorphic map $f_0^{(m)}: S^2\rightarrow\mathbb{C}P^{m}$ in homogeneous coordinates.
	\end{lemma}
	
	As to the second fundamental form $B$ of $\phi$,  by \eqref{eq:2.7} and a series of calculations, we obtain
	$$\left\{
	\begin{array} {l}
	\partial\phi=\frac{1}{|f^{(m)}_0|^2}[\overline{f}^{(m)}_0(\overline{f}^{(m)}_{1})^{\ast}+f^{(m)}_{1}f_0^{(m)\ast}],
	\\ A_z=\frac{1}{|f^{(m)}_0|^2}[\overline{f}_0^{(m)}(\overline{f}^{(m)}_{1})^{\ast}-f^{(m)}_{1}f_0^{(m)\ast}],
	\\ P=\frac{1}{2|f^{(m)}_1|^2}[\overline{f}_0^{(m)}(\overline{f}^{(m)}_{2})^{\ast}-f^{(m)}_{2}f_0^{(m)\ast}].
	\end{array} \right.$$ From this we derive the following useful relation
	$$\|B\|^2=2\frac{\delta_1^{(m)}}{\delta_0^{(m)}}-2\frac{  |\langle f_0^{(m)}, \overline f_2^{(m)}\rangle |^2}{|f^{(m)}_{1}|^4}.$$

	Set $$G_W:=\{U \in U(6)|U^TU=W\},$$ 
	in the following we shall characterize $\phi$ explicitly by virtue of Lemma 4.1, and prove the following property.

	\begin{lemma}
		Let $\phi: S^2\rightarrow G(2,6;\mathbb{R})$ be a linearly full
		reducible harmonic map with finite isotropy order $r$ and Gauss curvature  $K$. Suppose that $K$ is constant,  then $r=1$ and, up to an isometry
		of $G(2,6;\mathbb{R})$,  $\phi$ belongs to one of the following cases.
		\\ (1)  $\underline{\phi} = \underline{\overline{U}}\overline{\underline{V}}^{(3)}_0\oplus
		\underline{U}\underline{V}^{(3)}_0$ with $K=\frac{2}{3}$ for some $U\in G_W$, where W has the form \eqref{eq:4.13};
		\\ (2)  $\underline{\phi} = \underline{\overline{U}}\overline{\underline{V}}^{(2)}_0\oplus
		\underline{U}\underline{V}^{(2)}_0$ with $K= 1$  for some $U\in G_W$, where W has the form \eqref{eq:4.15}.
		In each of these two cases, there are many different types of W, thus exist different  $U \in U(6)$ such that $UV^{(m)}_0 (m=2, 3)$ are linearly full in $G(2,6;\mathbb{R})$, and they are not $SO(6)-$ equivalent.
	\end{lemma}
	\begin{proof} 
		According to above discussion, $r=1$ and $2 \leq m \leq 5$, here we deal with the four cases $m= 2, 3, 4, 5$ respectively.

		\textbf{(1) $m=5$.}
		
		Firstly we discuss  this case  and prove $m \neq 5$. To do this,  let us  assume that there exists  a linearly full
		reducible harmonic map $\underline{\phi} =
		\underline{\overline{f}}_0^{(5)} \oplus \underline{f}_0^{(5)}: S^2 \rightarrow G(2,6;\mathbb{R})$ with constant curvature and finite isotropy order,  then,  $f_0^{(5)} \in H^0_{6}$ and it can be obtained by Lemma 4.1.

		In (1) of Lemma 4.1, choose 
		$F_0(z)=\left(\begin{array}{cccccccccc}   a_{00}&  a_{01}&  a_{02} &  a_{03}
		\\a_{10}&  a_{11}&  a_{12} &  a_{13}
		\\ a_{20}&  a_{21}&  a_{22} &  a_{23}
		\\a_{30}&  a_{31}&  a_{32} &  a_{33}
		\end{array}\right)\left(\begin{array}{cccccccccc}   1\\z
		\\ z^2
		\\z^3
		\end{array}\right) \triangleq A \left(\begin{array}{cccccccccc}   1\\z
		\\ z^2
		\\z^3
		\end{array}\right)$,  
		where $A$ is a constant matrix with $\langle    F_0(z), \overline F_0(z)\rangle\neq 0$, then  using (2) of Lemma 4.1 we write  $H(z)$ in the form
		$$\left(\begin{array}{cccccccccc}   a_{00}z+\frac{1}{2}a_{01}z^2+\frac{1}{3}a_{02}z^3+\frac{1}{4}a_{03}z^4
		\\  a_{10}z+\frac{1}{2}a_{11}z^2+\frac{1}{3}a_{12}z^3+\frac{1}{4}a_{13}z^4
		\\  a_{20}z+\frac{1}{2}a_{21}z^2+\frac{1}{3}a_{22}z^3+\frac{1}{4}a_{23}z^4
		\\  a_{30}z+\frac{1}{2}a_{31}z^2+\frac{1}{3}a_{32}z^3+\frac{1}{4}a_{33}z^4
		\end{array}\right),$$ which gives
		\begin{equation}
			\langle H(z),   \overline{H}(z)\rangle = \sum_{i=0}^3(a_{i0}z+\frac{1}{2}a_{i1}z^2+\frac{1}{3}a_{i2}z^3+\frac{1}{4}a_{i3}z^4)^2. \label{eq:4.5}\end{equation}
		This relation together with (3) of Lemma 4.1 and the fact that $F_1(z)$  represents a holomorphic map of $S^2$ in $\mathbb{C}P^{5}$ show  that  coefficients of $z^6, \ z^7$ and $z^8$ in \eqref{eq:4.5} are all vanish, which can be expressed by
		\begin{equation}
			\frac{1}{9}\sum_{i=0}^{3}(a_{i2})^2+\frac{1}{4}\sum_{i=0}^{3}a_{i1}a_{i3}=0, \quad \sum_{i=0}^{3}a_{i2}a_{i3}=0, \quad  \sum_{i=0}^{3}(a_{i3})^2=0, \label{eq:4.6}\end{equation}
		and
		it is reasonable to put
		$$
		\langle H(z),   \overline{H}(z)\rangle = A_2\sqrt{10}z^2+A_3\sqrt{10}z^3+A_4\sqrt{5}z^4+A_5z^5$$
		for convenience, where $A_2, \ A_3, \ A_4$ and $A_5$ are constant, then it can be clearly seen that 
		$$F_1(z)=\left(\begin{array}{cccccccccc}   
		2a_{00}z+a_{01}z^2+\frac{2}{3}a_{02}z^3+\frac{1}{2}a_{03}z^4
		\\ 2a_{10}z+a_{11}z^2+\frac{2}{3}a_{12}z^3+\frac{1}{2}a_{13}z^4
		\\ 2a_{20}z+a_{21}z^2+\frac{2}{3}a_{22}z^3+\frac{1}{2}a_{23}z^4
		\\ 2a_{30}z+a_{31}z^2+\frac{2}{3}a_{32}z^3+\frac{1}{2}a_{33}z^4
		\\ 1-A_2\sqrt{10}z^2-A_3\sqrt{10}z^3-A_4\sqrt{5}z^4-A_5z^5
		\\ \sqrt{-1}(1+A_2\sqrt{10}z^2+A_3\sqrt{10}z^3+A_4\sqrt{5}z^4+A_5z^5)
		\end{array}\right)=f^{(5)}_0=UV^{(5)}_0$$
		with 
		\begin{equation}U=\left(\begin{array}{cccccccccc}   
				0&\frac{2a_{00}}{\sqrt{5}}& \frac{a_{01}}{\sqrt{10}} &\frac{2a_{02}}{3\sqrt{10}}  & \frac{a_{03}}{2\sqrt{5}}  & 0
				\\ 0&\frac{2a_{10}}{\sqrt{5}}& \frac{a_{11}}{\sqrt{10}} &\frac{2a_{12}}{3\sqrt{10}}  & \frac{a_{13}}{2\sqrt{5}} & 0
				\\ 0&\frac{2a_{20}}{\sqrt{5}}& \frac{a_{21}}{\sqrt{10}} &\frac{2a_{22}}{3\sqrt{10}}  & \frac{a_{23}}{2\sqrt{5}}& 0
				\\ 0&\frac{2a_{30}}{\sqrt{5}}& \frac{a_{31}}{\sqrt{10}} &\frac{2a_{32}}{3\sqrt{10}}  & \frac{a_{33}}{2\sqrt{5}}& 0
				\\ 1&0& -A_2 &-A_3 &-A_4& -A_5
				\\ \sqrt{-1}&0& \sqrt{-1}A_2 &\sqrt{-1}A_3 &\sqrt{-1}A_4& \sqrt{-1}A_5
			\end{array}\right).\label{eq:4.7}\end{equation}
		Here it is importance to notice that this  $U$ satisfies  $UU^*=\mu I_{6\times 6}$ for some constant $\mu$  from our assumption that $\phi$ is  of constant curvature, which means that the $F_1(z)$  constructed above is of constant curvature.

		Set $W=U^TU \triangleq  \left(\begin{array}{cccccccccc}   
		w_{00}&w_{01}& w_{02}&w_{03}& w_{04}&w_{05}
		\\ w_{10}&w_{11}& w_{12}&w_{13}& w_{14}&w_{15}
		\\ w_{20}&w_{21}& w_{22}&w_{23}& w_{24}&w_{25}
		\\ w_{30}&w_{31}& w_{32}&w_{33}& w_{34}&w_{35}
		\\ w_{40}&w_{41}& w_{42}&w_{43}& w_{44}&w_{45}
		\\ w_{50}&w_{51}& w_{52}&w_{53}& w_{54}&w_{55}
		\end{array}\right)$.  By the standard expression of $V^{(5)}_0$ given in Section 3, we get $V^{(5)}_0 V_0^{(5)T}$ is a polynomial matrix in $z$ and $\overline{z}$. Using the method of indeterminate coefficients,  \eqref{eq:4.6} gives the relation 
		\begin{equation} w_{34}=w_{44}=0 , \quad  w_{33}+\sqrt{2}w_{24}=0,\label{eq:4.8}\end{equation}
		and by \eqref{eq:4.7} we conclude 
		$w_{ij}=w_{ji}$ for any $0\leqslant i, j \leqslant 5$.

		Furthermore using \eqref{eq:4.2} we have the relation $$\langle F_1(z), \overline F_1(z)\rangle = \textrm{tr} WV^{(5)}_0 V^{(5)T}_0= 0.$$
		A series calculations give
		\begin{equation} w_{00}=w_{01}=w_{45}=w_{55}= 0, \quad  2\sqrt{10}w_{02}+5w_{11}= 0, \label{eq:4.9}\end{equation}
		\begin{equation}    w_{03}+\sqrt{5}w_{12}= 0, \quad  2\sqrt{10}w_{35}+5w_{44}= 0,  \quad w_{25}+\sqrt{5}w_{34}= 0, \label{eq4.10}\end{equation}
		\begin{equation}  w_{04}+\sqrt{10}w_{13}+\sqrt{5}w_{22}= 0,  \quad w_{15}+\sqrt{10}w_{24}+\sqrt{5}w_{33}= 0,  \label{eq:4.11}\end{equation}
		\begin{equation}  w_{05}+5w_{14}+10w_{23}= 0.\label{eq:4.12}\end{equation}
		
		Combing \eqref{eq:4.8}-\eqref{eq:4.11} and using the property of the unitary matrix, this is a straightforward computation
		$$U^TU \triangleq  \left(\begin{array}{cccccccccc}   
		0&0& 0 &0& 0 & w_{05}
		\\ 0&0& 0 &0& w_{14} & 0
		\\ 0&0& 0 &w_{23}& 0 & 0
		\\ 0&0& w_{32} &0& 0 & 0
		\\ 0&w_{41}& 0 &0& 0 & 0
		\\ w_{50}&0& 0&0&0&0
		\end{array}\right)$$ with $|w_{05}|=|w_{14}|=|w_{23}|$,  which contradicts \eqref{eq:4.12}. So  $m \neq 5$ is proved.

		\textbf{(2) $m=4$.}
		
		Analogous $m \neq 4$ can be proved by using the same method as above.

		\textbf{(3) $m=3$.}
		
		From the fact that 
		$$V_0^{(3)}=(1, \sqrt{3}z, \sqrt{3}z^2, z^3)^T,$$
		and then using the method of indeterminate coefficients,  \eqref{eq:4.4} gives 
		\begin{equation}W=U^TU \triangleq  \left(\begin{array}{cccccccccc}   
				0&0& w_{02}&w_{03}& w_{04}&w_{05}
				\\ 0&-\frac{2\sqrt{3}}{3}w_{02}& -\frac{1}{3}w_{03}&w_{13}& w_{14}&w_{15}
				\\ w_{02}&-\frac{1}{3}w_{03}& -\frac{2\sqrt{3}}{3}w_{13}&0& w_{24}&w_{25}
				\\ w_{03}&w_{13}& 0&0& w_{34}&w_{35}
				\\ w_{04}&w_{14}& w_{24}&w_{34}& w_{44}&w_{45}
				\\ w_{05}&w_{15}& w_{25}&w_{35}& w_{45}&w_{55}
			\end{array}\right).\label{eq:4.13}\end{equation}
		There are many different such type of $W$, thus  with different $U$. In other words, we can find different $U$ to write $\underline{\phi} = \underline{\overline{U}}\overline{\underline{V}}^{(3)}_0\oplus
		\underline{U}\underline{V}^{(3)}_0$, and they are not congruent. Here we just give one example of them. Choose  
		$$W=\left(\begin{array}{cccccccccc}
		0 & 0 & 0 &1 & 0 & 0 \\
		0 & 0 & -\frac{1}{3} &0 & \frac{\sqrt{8}}{3}  & 0 \\
		0 & -\frac{1}{3}  & 0 &0 & 0 & \frac{\sqrt{8}}{3}\\
		1 & 0 & 0 &0 & 0 & 0 \\
		0 & \frac{\sqrt{8}}{3} & 0 &0 & 0 & \frac{1}{3} \\
		0 & 0 & \frac{\sqrt{8}}{3}  &0 & \frac{1}{3} & 0 
		\end{array}\right)$$ and  $$U
		=\left(\begin{array}{cccccccccc}
		\frac{1}{\sqrt{2}} & 0 & 0 &\frac{1}{\sqrt{2}}  & 0 & 0 \\
		\frac{\sqrt{-1}}{\sqrt{2}} & 0 & 0 &- \frac{\sqrt{-1}}{\sqrt{2}} & 0 & 0 \\
		0 &  \frac{1}{\sqrt{2}}   & -\frac{1}{3\sqrt{2}} &0& \frac{2}{3}  & 0\\
		0 &  \frac{\sqrt{-1}}{\sqrt{2}}   & \frac{\sqrt{-1}}{3\sqrt{2}} &0& -\frac{2\sqrt{-1}}{3}  & 0 \\
		0 & 0 & \frac{2}{3} &0 & \frac{1}{3\sqrt{2}} & \frac{1}{\sqrt{2}} \\
		0 & 0 & \frac{2\sqrt{-1}}{3} &0 & \frac{\sqrt{-1}}{3\sqrt{2}} & \frac{-\sqrt{-1}}{\sqrt{2}} 
		\end{array}
		\right).$$
		In this case  $\underline{\phi} = \underline{\overline{U}}\overline{\underline{V}}^{(3)}_0\oplus
		\underline{U}\underline{V}^{(3)}_0 = \overline{\underline{f}}^{(3)}_0\oplus
		\underline{f}^{(3)}_0$ has Gauss curvature $K=\frac{2}{3}$, where 
		\begin{equation}f^{(3)}_0=[(1+z^3, \  \sqrt{-1}(1-z^3),   \  \sqrt{3}z-\frac{z^2}{\sqrt{3}},   \  \sqrt{-1}(\sqrt{3}z+\frac{z^2}{\sqrt{3}}),   \ 
			\frac{\sqrt{8}}{\sqrt{3}}z^2,  \ 
			\frac{\sqrt{-8}}{\sqrt{3}}z^2)^T].\label{eq:4.14}\end{equation}
		Direct computations give
		$$\|B\|^2=\frac{8}{3}-\frac{32z\overline z}{9(1+z\overline z)^2}.$$

		\textbf{(4) $m=2$.}   
		
		Analogous,  by using   $V_0^{(2)}=(1, \sqrt{2}z, z^2)^T$,  we get the type of  $W=U^TU \in U(6)$ as follows
		\begin{equation}W=U^TU \triangleq  \left(\begin{array}{cccccccccc}   
				0&0& w_{02}&w_{03}& w_{04}&w_{05}
				\\ 0&-w_{02}& 0&w_{13}& w_{14}&w_{15}
				\\ w_{02}&0& 0&w_{23}& w_{24}&w_{25}
				\\ w_{03}&w_{13}& w_{23}&w_{33}& w_{34}&w_{35}
				\\ w_{04}&w_{14}& w_{24}&w_{34}& w_{44}&w_{45}
				\\ w_{05}&w_{15}& w_{25}&w_{35}& w_{45}&w_{55}
			\end{array}\right) \label{eq:4.15}\end{equation} with $w_{02} \neq 0$.   An example can be given by choosing
		
		$$W=\left(\begin{array}{cccccccccc}
		0 & 0 & \frac{1}{2}  &0 & 0 & -\frac{\sqrt{3}}{2}  \\
		0 & -\frac{1}{2}  & 0 &0 & -\frac{\sqrt{3}}{2}  & 0 \\
		\frac{1}{2}  & 0  & 0 &-\frac{\sqrt{3}}{2}  & 0 & 0\\
		0 & 0 & -\frac{\sqrt{3}}{2}  &0 & 0 & -\frac{1}{2}  \\
		0 & -\frac{\sqrt{3}}{2}  & 0 &0 & \frac{1}{2}  & 0 \\
		-\frac{\sqrt{3}}{2}  & 0 & 0  &-\frac{1}{2}  & 0 & 0 
		\end{array}\right), $$  
		and

		$$ U
		=\left(\begin{array}{cccccccccc}
		\frac{1}{\sqrt{2}} & 0 & \frac{1}{2\sqrt{2}}  &0  & 0 & -\frac{\sqrt{3}}{2\sqrt{2}} \\
		\frac{\sqrt{-1}}{\sqrt{2}} & 0 & - \frac{\sqrt{-1}}{2\sqrt{2}}  &0& 0 & \frac{\sqrt{-3}}{2\sqrt{2}}  \\
		0 &  0   & -\frac{\sqrt{-3}}{2\sqrt{2}}  &-\frac{\sqrt{-1}}{\sqrt{2}} &0  & -\frac{\sqrt{-1}}{2\sqrt{2}} \\
		0 &  0  & -\frac{\sqrt{3}}{2\sqrt{2}}  &\frac{1}{\sqrt{2}} &0  & -\frac{1}{2\sqrt{2}}  \\
		0 & \frac{1}{2} & 0&0 & -\frac{\sqrt{3}}{2} & 0\\
		0 & -\frac{\sqrt{-3}}{2} & 0 &0 & -\frac{\sqrt{-1}}{2} & 0
		\end{array}\right]. $$
		In this case  $\underline{\phi} = \underline{\overline{U}}\overline{\underline{V}}^{(2)}_0\oplus
		\underline{U}\underline{V}^{(2)}_0 = \overline{\underline{f}}^{(2)}_0\oplus
		\underline{f}^{(2)}_0$ has Gauss curvature $K=1$, where 
		\begin{equation}f^{(2)}_0= UV^{(2)}_0=[(1+\frac{z^2}{2}, \  \sqrt{-1}(1-\frac{z^2}{2}),   \  -\frac{\sqrt{-3}}{2}z^2,   \  -\frac{\sqrt{3}}{2}z^2,   \ 
			z,  \ 
			-\sqrt{-3}z)^T).\label{eq:4.16}\end{equation}
		As to the second fundamental form $B$ of $\phi$,  by an straightforward  computation, we obtain $$\|B\|^2=\frac{3}{2}.$$ 
		In summary we get the conclusion.
	\end{proof}

	Let
\begin{equation}
W=
\left(
\begin{array}{ccc|ccc}
&  &  & w_{03} & w_{04} & w_{05} \\
& \text{{\huge{0}}} &  & w_{13} & w_{14} & w_{15} \\
&  &  & w_{23} & w_{24} & w_{25} \\
w_{03} & w_{13} & w_{23} & w_{33} & w_{34} & w_{35}
\\ w_{04} & w_{14} & w_{24} & w_{34} & w_{44} & w_{45}
\\w_{05} & w_{15} & w_{25} & w_{35} & w_{45} & w_{55}
\end{array}
\right),\label{eq:4.17}\end{equation}	
	for general linearly full reducible harmonic map with constant curvature from $S^2$ to $G(2,6;\mathbb{R})$,  by Lemma 4.2 and (\cite{10}, Proposition 3.5) we have
	\begin{prop}
		Let $\phi: S^2\rightarrow G(2,6;\mathbb{R})$ be a linearly full
		reducible harmonic map with Gauss curvature  $K$. Suppose that $K$ is constant,  then, up to an isometry
		of $G(2,6;\mathbb{R})$,  $\phi$ belongs to one of the following cases.
		\\ (1)  $\underline{\phi} = \underline{\overline{U}}\overline{\underline{V}}^{(3)}_0\oplus
		\underline{U}\underline{V}^{(3)}_0$ with $K=\frac{2}{3}$ for some $U\in G_W$, where W has the form \eqref{eq:4.13};
		\\ (2)  $\underline{\phi} = \underline{\overline{U}}\overline{\underline{V}}^{(2)}_0\oplus
		\underline{U}\underline{V}^{(2)}_0$ with $K= 1$  for some $U\in G_W$, where W has the form \eqref{eq:4.15};
		\\ (3)  $\underline{\phi} = \underline{\overline{U}}\overline{\underline{V}}^{(2)}_0\oplus
		\underline{U}\underline{V}^{(2)}_0$ with $K=1$ for some $U\in G_W$, where W has the form \eqref{eq:4.17};
		\\ (4)  $\underline{\phi} = \underline{U}\underline{V}^{(4)}_2\oplus \underline{c}_0$ with $K=\frac{1}{3}$  for some $U\in U(5)$ and ${c}_0 = (0, 0, 0, 0, 0, 1)^{T}$.
		\\ In (1) (2) and (3), there are many different types of W, thus exist different  $U \in U(6)$ such that $UV^{(m)}_0 (m=2, 3)$ are linearly full in $G(2,6;\mathbb{R})$, and they are not $SO(6)-$ equivalent.
	\end{prop}

	In Proposition 4.3,  (1)   gives us a non-homogeneous constant curved minimal two-sphere in $Q_4$. (2) and (3) stand for two different holomorphic curves from the Riemann sphere into  $Q_4$  whose curvature are both equal to $1$, which illustrates conformal minimal two-spheres of constant curvature in complex hyperquadric $Q_{n}$ are in general not  equivalent, 
	constracting to the fact that generic isometric complex submanifolds in a Kaehler manifold are congruent.  They show us that the case of  $Q_{n}$ is very complicated, and it is very difficult for classifications of conformal minimal two-spheres of constant curvature in a complex hyperquadric $Q_{n}$.

	Here the type of $U$ in (3) of  Proposition 4.3  may be chosen as
	\begin{equation}
		U=U_0=\left(
		\begin{array}{ccc|c}
				\frac{1}{\sqrt2} & 0 & 0 &  \\
				\frac{\sqrt{-1}}{\sqrt2} & 0 & 0 & \\
				0 & \frac{1}{\sqrt2} & 0 &  \\
				0 & \frac{\sqrt{-1}}{\sqrt2} & 0 & \text{{\huge{*}}}\\
				0 & 0 & \frac{1}{\sqrt2} & \\
				0 & 0 & \frac{\sqrt{-1}}{\sqrt2} &
			\end{array}
			\right), \label{eq:4.18}
	\end{equation} 
	and 
	$U$ in (4) of  Proposition 4.3  may be chosen as
	\begin{equation}U= U_1
		=\left[\begin{array}{cccccccccc} \frac{1}{\sqrt{2}} & 0 &
			0 & 0 & \frac{1}{\sqrt{2}} \\
			\frac{\sqrt{-1}}{\sqrt{2}} & 0 & 0 &  0 &
			-\frac{\sqrt{-1}}{\sqrt{2}}
			\\ 0 & \frac{1}{\sqrt{2}} & 0 & -\frac{1}{\sqrt{2}} & 0\\0 &
			\frac{\sqrt{-1}}{\sqrt{2}}& 0 & \frac{\sqrt{-1}}{\sqrt{2}} & 0 \\  0 & 0 &
			1 & 0 & 0
		\end{array}\right]. \label{eq:4.19}\end{equation} 
	Then, up to an isometry
	of $G(2,6;\mathbb{R})$,  either $$\underline{\phi} = \underline{\overline{U}}_0\overline{\underline{V}}^{(2)}_0\oplus
	\underline{U}_0\underline{V}^{(2)}_0=\underline{\overline{f}}^{(2)}_0 \oplus \underline{f}^{(2)}_0$$ with \begin{equation}f^{(2)}_0 = [(1, \sqrt{-1}, \sqrt{2}z, \sqrt{-2}z, z^2, \sqrt{-1}z^2)^{T}]; \label{eq:4.20}\end{equation} or
	$$\underline{\phi} = \underline{U}_1\underline{V}^{(4)}_2\oplus
	\underline{c}_0=\underline{f}^{(4)}_2 \oplus \underline{c}_0$$ with $f^{(4)}_2$ of the follwing expression
	\begin{equation}[(z^2+\overline{z}^2, \sqrt{-1}(\overline{z}^2-z^2),  (z+\overline{z})(|z|^2-1),   \sqrt{-1}(\overline{z}-z)(|z|^2-1),  \frac{1-4|z|^2+|z|^4}{\sqrt{3}})^{T}]. \label{eq:4.21}\end{equation}
	By Theorem 1.1 of \cite{11}, these two maps shown in  \eqref{eq:4.20} and \eqref{eq:4.21}  are all of parallel second fundamental form.

	\section{Irreducible harmonic maps of constant curvature}

	In this section we shall discuss irreducible harmonic map $\phi:
	S^2\rightarrow G(2,6;\mathbb{R})$ of isotropy order $r$.  If $\phi$ has finite isotropy order, then $r=1$ by (\cite{1}, Proposition
	2.8 and Lemma 2.15), and(\cite{10}, Proposition 4.2) implies that
	
	\begin{prop}
		The map   $\phi: S^2\rightarrow G(2,6;\mathbb{R})$  is a linearly
		full irreducible harmonic map with finite isotropy order  if and only if $\underline{\phi}=
		\underline{\overline{V}} \oplus \underline{V}$ with $V = f_1^{(m)} +
		x_0\overline{f}_0^{(m)}$, where $f_0^{(m)}$ is a holomorphic map satisfying  $ \left\{  \begin{array}{l} \langle \overline{f}_0^{(m)},   f_{3}^{(m)} \rangle=0\\
		\langle \overline{f}_0^{(m)}, f_{4}^{(m)} \rangle\neq 0
		\end{array}\right.$,
		and the corresponding coefficient $x_0$ satisfies  equation
		$\partial \overline{x}_0 + \overline{x}_0\partial \log |f_0^{(m)}|^2
		=0, \ m=4 $ or $5$, here $\underline{f}_0^{(m)},...,\underline{f}_m^{(m)}: S^2\rightarrow \mathbb{C}P^{m}$ is a linearly full harmonic sequence in $\mathbb{C}P^{m} \subset  \mathbb{C}P^{5}$.
	\end{prop}

	Furthermore, if  $\phi$ is of constant curvature, by (\cite{10}, Proposition 4.3), since $n=6$ is even, there doesn't exist linearly full totally unramified irreducible conformal minimal immersion of $S^2$ in $G(2,6;\mathbb{R})$ with constant curvature and finite isotropy order. In the following, we only consider the (strongly) isotropic ones. To characterize such $\phi$ we first state one of  Burstall and Wood'
	results (\cite{4}, Theorem 2.4 and Proposition 3.7, 3.8) as follows:
	
	\begin{lemma}[Special case of {\cite{4}}]
		Let $\underline{\phi}:S^2\rightarrow G(2,n;\mathbb{C})$ be a (strongly) isotropic harmonic subbundle of $\underline{\mathbb{C}}^n$,
		\\(i) If $\underline{\phi}$ is $\partial^{'}$-irreducible.  Let $\underline{\alpha}$ be a holomorphic subbundle of $\underline{\phi}$ such that $\underline{\alpha}\subset ker A'_{\phi^{\perp}}\circ A'_{\phi}$, then,  the bundle $\underline{\widetilde{\phi}}$ given by $\underline{\widetilde{\phi}}=\underline{\phi}\bigcap \underline{\alpha}^{\bot}\oplus \underline{Im}(A'_{\phi}|\underline{\alpha})$ is harmonic;
		\\(ii) If $\underline{\phi}$ is harmonic with $\partial{'}\underline{\phi}$ of rank one and  $A''_{\phi}(\underline{ker}A^{'\perp}_{\phi})\neq 0$.  Let $\underline{\alpha}=\underline{ker}A'_{\phi}$, then backward replacement of $\underline{\beta}=\underline{\alpha}^{\perp}\bigcap \underline{\phi}$ produces a new harmonic map $\underline{\widetilde{\phi}}=\underline{\alpha}\oplus \underline{Im}(A''_{\phi}|\underline{\beta}): S^2 \rightarrow G(2,n;\mathbb{C})$, where $\partial'\widetilde{\underline{\phi}}=\underline{\beta},  \ \partial^{(i)}\widetilde{\underline{\phi}}= \partial^{(i-1)}\underline{\phi}$ for $i\geq 2$;\\(iii)
		If $\underline{\phi}$ is harmonic with $\partial{'}\underline{\phi}$ of rank one and $A''_{\phi}(\underline{ker}A^{'\perp}_{\phi})=0$. Then either (a) there is an antiholomorphic map $g:S^2\rightarrow \mathbb{C}P^{n-1}$ and $\underline{\phi}=\partial^{(-r)}\underline{g}\oplus \partial^{(-r-1)}\underline{g}$ for some integer $r\geq 0$, (it can be shown that $\phi$ is a Frenet pair) or (b) there are maps $g,h:S^2 \rightarrow \mathbb{C}P^{n-1}$ antiholomorphic and holomorphic respectively such that $\partial'\underline{h}\perp \underline{g}$ and $\underline{\phi}=\underline{g}\oplus \underline{h}$, i.e. $\underline{\phi}$ is a mixed pair.
	\end{lemma}

	$A^{'}_{\phi}$ and $A^{''}_{\phi}$ shown in Lemma 5.2 are vector bundle morphisms from $\underline{\phi}$ to $\underline{\phi}^{\bot}$, they are defined by $A^{'}_{\phi}(v)=\pi_{\phi^{\bot}}(\partial v)$ and $A^{''}_{\phi}(v)=\pi_{\phi^{\bot}}(\overline{\partial} v)$ respectively for some $v\in\mathbb{C}^{\infty}(\underline{\phi})$ (cf. \cite{1, 4}).

	Let $\phi: S^2\rightarrow G(2,6;\mathbb{R})$ be a linearly full
	irreducible harmonic map with isotropy order $r = \infty$. In the following we characterize
	$\phi$ explicitly by virtue of Lemma 5.2.
	
	Since $\phi$ is a  (strongly) isotropic irreducible harmonic map from $S^2$ to $G(2,6;\mathbb{R})$, it belongs to the following harmonic sequence
	
	$$0\stackrel{\partial{''}} {\longleftarrow}
	\underline{\phi}_{-1}\stackrel{\partial{''}} {\longleftarrow}
	\underline{\phi}\stackrel {\partial{'}}
	{\longrightarrow}\underline{\phi}_1 \stackrel{\partial{'}}
	{\longrightarrow}0,$$
	where $\underline{\phi}_{-1}=\overline{\underline{\phi}}_1$ is of rank 2, and $\underline{\phi}$ can be expressed by $\underline{\phi}=\underline{\overline{X}}\oplus \underline{X}$, here $\{\underline{\overline{X}}, \underline{X}\}$ is the unique unordered holomorphic subbundles of rank one of $\underline{\phi}$ (cf. \cite{1}).

	Let $$\underline{Y}=A'_{\phi}|\underline{X}, \ Z = Y^{\bot} \cap {\phi}_1,$$ then we have $\overline{X},X,\overline{Y},Y,\overline{Z},Z$ are mutually orthogonal and $A'_{\phi^{\bot}}|Y=0$, i.e. $$\underline{X}\subset ker A'_{\phi^{\perp}}\circ A'_{\phi}.$$ Then by (i) of Lemma 5.2, $$\widetilde{\underline{\phi}}=\underline{\overline{X}}\oplus\underline{Y}$$ is harmonic. Through a straightforward computation, $\widetilde{\underline{\phi}}$ belongs to the following harmonic sequence
	
	\begin{equation}0\stackrel{\partial{''}} {\longleftarrow}
		\overline{\underline{Z}}\stackrel{\partial{''}}
		{\longleftarrow}\underline{X}\oplus\underline{\overline{Y}}
		\stackrel{\partial{''}} {\longleftarrow}
		\widetilde{\underline{\phi}}=\underline{\overline{X}}\oplus\underline{Y}\stackrel {\partial{'}}
		{\longrightarrow}\underline{Z}\stackrel{\partial{'}}
		{\longrightarrow}0,  \label{eq:5.1}\end{equation} which implies that subbundle $\underline{Z}$ is harmonic and antiholomorphic, without loss of generality, we assume $\underline{Z}=\underline{f}^{(m)}_m$, it is a linearly full harmonic map from $S^2$ to $\mathbb{C}P^m$  for some $m<6$ and belongs to the following harmonic sequence
	$$0\stackrel{\partial{'}} {\longrightarrow}
	\underline{f}^{(m)}_0\stackrel{\partial{'}}
	{\longrightarrow}\underline{f}^{(m)}_1
	\stackrel{\partial{'}} {\longrightarrow}
	\cdots\stackrel {\partial{'}}
	{\longrightarrow}\underline{f}^{(m)}_m\stackrel{\partial{'}}
	{\longrightarrow}0,$$ where $\overline{\partial}f^{(m)}_0=0$ and $f^{(m)}_0, f^{(m)}_1, \ldots, f^{(m)}_m$ satisfy \eqref{eq:2.9} and \eqref{eq:2.10}.
	
	By \eqref{eq:5.1},  $\underline{f}^{(m)}_{m-1}$ is a subbundle with rank one of $\underline{\widetilde{\phi}}$, let $W = f^{(m)\bot}_{m-1} \cap \widetilde{\phi}$, then \eqref{eq:5.1} can be rewritten as
	
	\begin{equation}0\stackrel{\partial{''}} {\longleftarrow}
		\overline{\underline{f}}^{(m)}_m\stackrel{\partial{''}}
		{\longleftarrow}\underline{\overline{W}}\oplus\underline{\overline{f}}^{(m)}_{m-1}
		\stackrel{\partial{''}} {\longleftarrow}
		\widetilde{\underline{\phi}}= \underline{W}\oplus\underline{f}^{(m)}_{m-1}\stackrel {\partial{'}}
		{\longrightarrow}\underline{f}^{(m)}_m\stackrel{\partial{'}}
		{\longrightarrow}0.  \label{eq:5.2}\end{equation} 
	Here $\overline{W},W,\overline{f}^{(m)}_{m-1}, f^{(m)}_{m-1}, \overline{f}^{(m)}_m,f^{(m)}_m$ are mutually orthogonal and $\underline{W}$ is a holomorphic subbundle of $\widetilde{\underline{\phi}}$, it satisfies $A'_{\widetilde{\phi}}|W=0$ and $A''_{\widetilde{\phi}}|f^{(m)}_{m-1}\neq 0$, i.e. $$\underline{W}=\underline{ker}A'_{\widetilde{\phi}}, \quad A''_{\widetilde{\phi}}(\underline{ker}A^{'\perp}_{\widetilde{\phi}})\neq 0.$$ Then by (ii) of Lemma 5.2, the backward replacement of $f^{(m)}_{m-1}$ produces a new harmonic map $$\underline{\varphi}=\underline{\overline{W}}\oplus \underline{W}: S^2\rightarrow G(2,6;\mathbb{R}),$$ it derives a harmonic sequence as follows
	
	\begin{equation}0\stackrel{\partial{''}} {\longleftarrow}
		\overline{\underline{f}}^{(m)}_m\stackrel{\partial{''}} {\longleftarrow}
		\overline{\underline{f}}^{(m)}_{m-1}\stackrel{\partial{''}} {\longleftarrow}
		\underline{\varphi}=\underline{\overline{W}}\oplus\underline{W}\stackrel {\partial{'}}
		{\longrightarrow}\underline{f}^{(m)}_{m-1} \stackrel {\partial{'}}
		{\longrightarrow}\underline{f}^{(m)}_m \stackrel{\partial{'}}
		{\longrightarrow}0.  \label{eq:5.3}\end{equation}

	Then we prove the following result.
	
	\begin{prop}
		$m=2$ if  $\phi: S^2\rightarrow G(2,6;\mathbb{R})$ is a linearly full totally unramified
		irreducible (strongly) isotropic  harmonic map of constant curvature. 
	\end{prop}
	\begin{proof}
		Suppose $\phi:S^2 \rightarrow G(2,6;\mathbb{R})$ is of constant curvature $K$.  From the above discussion, we choose local frame
		$$e_1 = \frac{\overline{X}}{|X|},  \  e_2 = \frac{X}{|X|}, \
		e_3 = \frac{Y}{|Y|},  \ e_4 = \frac{f^{(m)}_m}{|f^{(m)}_m|}, \ e_5 =
		\frac{\overline{Y}}{|Y|},  \  e_6 = \frac{\overline{f}^{(m)}_m}{|f^{(m)}_m|},$$
		here the local frame we choose  is unitary
		frame. Set $$W_0 = (e_1, e_2), \quad W_1 = (e_3, e_4), \quad W_{-1} = (e_5,
		e_6),$$ then by \eqref{eq:2.5}, we obtain
		$$\Omega_{-1} = -\left(\begin{array}{ccccccc}
		\frac{\langle
			\partial X,  Y\rangle }{|X||Y|} & 0 \\
		\frac{\langle
			\partial \overline{X},  Y\rangle }{|X||Y|}&  \frac{\langle\partial \overline{X}, f^{(m)}_{m}\rangle }{|X||f^{(m)}_{m}|}  \end{array}\right),   \quad \Omega_{0} =
		\left(\begin{array}{ccccccc}\frac{\langle
			\partial \overline{X},  Y\rangle }{|X||Y|} & \frac{\langle
			\partial X,  Y\rangle }{|X||Y|} \\
		\frac{\langle\partial \overline{X}, f^{(m)}_{m}\rangle }{|X||f^{(m)}_{m}|} & 0
		\end{array}\right).$$
		This together with  equation $L_{i} =\textrm{tr}
		(\Omega_{i} \Omega^*_{i})$ implies that
		\begin{equation}L_0 =L_{-1}= \frac { \langle
				\partial \overline{X} , Y\rangle \langle Y, \partial \overline{X} \rangle }{|X|^2|Y|^2} +\frac { \langle
				\partial X , Y\rangle \langle Y, \partial X \rangle }{|X|^2|Y|^2}+\frac { \langle
				\partial \overline{X} , f^{(m)}_{m}\rangle \langle f^{(m)}_{m}, \partial \overline{X} \rangle }{|X|^2|f^{(m)}_{m}|^2}. \label{eq:5.4}\end{equation}

		On the one hand, since $\phi$ is totally unramified,  it follows from
		\eqref{eq:3.2} and \eqref{eq:5.4} that \begin{equation} \delta_{-1} = \delta_0, \  \delta_1 = 0.\label{eq:5.5}\end{equation}
		On the other hand, by Lemma 3.3 we have
		\begin{equation}\delta_1 - 2\delta_0 + \delta_{-1} = -4, \label{eq:5.6}\end{equation}
		where $\delta_{i} = \frac{1}{2\pi\sqrt{-1} }\int_{S^2}
		L_{i}d{\overline{z}} \wedge dz,\ i =-1, 0, 1$. Substitution of \eqref{eq:5.5} in \eqref{eq:5.6} yields
		$$\delta_0 = 4.$$
		This formula and the fact that $\phi$ is of constant curvature enable us to set $K=\frac{1}{2}$, and complex coordinate $z$
		on $\mathbb{C} =S^2 \backslash \{pt\}$ can be chosen so that  the induced  metric
		$ds^2 = 2L_0dz d{\overline{z}}$ of
		$\phi$ is given by $$ds^2 =
		\frac{8}{(1+z\overline{z})^2} dz d{\overline{z}},$$ where
		\begin{equation}L_0
			=\frac{4}{(1+z\overline{z})^2}. \label{eq:5.7}\end{equation}
		
		From the fact that $\phi$ is irreducible and (strongly) isotropic, the harmonic sequence it derived can be rewritten as 
		$$0\stackrel{\partial{'}} {\longrightarrow}
		\underline{\phi}_{-1}\stackrel{\partial{'}} {\longrightarrow}
		\underline{\phi}\stackrel {\partial{'}}
		{\longrightarrow}\underline{\phi}_1 \stackrel{\partial{'}}
		{\longrightarrow}0,$$ here $\underline{\phi}_{-1}$ is a holomorphic curve with constant curvature $1$.  In the following we shall prove that for any holomorphic section of $\underline{\phi}_{-1}$, its degree will be $2$, i.e. $m=2$. 
		
		Let $f(z)$ and $g(z)$ be two holomorphic sections such that $\underline{\phi}_{-1}=\Pi \{ f(z), g(z) \}$.  Pl\"{u}cker imbedding {\cite{14}}
		$$[F]=[f(z) \wedge g(z) ]: S^2 \rightarrow \mathbb{C}P^{14}$$
		is a nowhere zero holomorphic curve. This is a holomorphic isometry, i.e., 
		$$[F]^*ds^2_{ \mathbb{C}P^{14}}=\phi_{-1}^*ds^2_{ G(2,6;\mathbb{R})}.$$
		Then $\phi_{-1}$ and $[F]$ have the same curvature $1$.  Set 
		$$\left(\begin{array}{ccccccc}
		f(z)  \\
		g(z)  \end{array}\right)=
		\left(\begin{array}{ccccccc}
		1,&0,& a(z),& b(z),& c(z),& d(z) \\
		0,& 1,& p(z),& q(z),& r(z), & s(z) \end{array}\right).$$
		By Lemma 3.4, there exists the unitary matrix $U\in U(15)$ such that
		$$f(z)\wedge g(z)={V}_0^{(4)}U,$$
		where  $V_0^{(4)}$ is the Veronese curve in $\mathbb{C}P^4$ given in Section 3 (adding zeros to $V_0^{(4)}$ such that $V_0^{(4)} \in \mathbb{C}^{15}$). Thus it is very evident  that
		\begin{equation}| f(z)  \wedge g(z)  |^2=(1+z\overline{z})^4,\label{eq:5.8}\end{equation}
		which shows  $a(z),b(z),c(z),d(z)$ and $p(z),q(z),r(z),s(z)$ are all polynomials in $z$ with degree $<5$, i.e.,
		$$a(z)=a_1z+a_2z^2+a_3z^3+a_4z^4, \ b(z)=b_1z+b_2z^2+b_3z^3+b_4z^4,$$
		$$c(z)=c_1z+c_2z^2+c_3z^3+c_4z^4, \ d(z)=d_1z+d_2z^2+d_3z^3+d_4z^4,$$
		$$p(z)=p_1z+p_2z^2+p_3z^3+p_4z^4, \ q(z)=q_1z+q_2z^2+q_3z^3+q_4z^4,$$
		$$r(z)=r_1z+r_2z^2+r_3z^3+r_4z^4, \ s(z)=s_1z+s_2z^2+s_3z^3+s_4z^4.$$
		By \eqref{eq:5.8},  $a(z)q(z)-b(z)p(z), \  a(z)r(z)-c(z)p(z), \ a(z)s(z)-d(z)p(z), \ b(z)r(z)-c(z)q(z)$, $b(z)s(z)-d(z)q(z) $ and  $c(z)s(z)-d(z)r(z) $ are also polynomials in $z$ with degree $<5$, then
		$$\frac{a_4}{p_4}=\frac{b_4}{q_4}=\frac{c_4}{r_4}=\frac{d_4}{s_4}.$$ Hence there exist the U(4)-transformation $I_2 \times U_4$ so that
		$$(a_4, b_4, c_4, d_4)U_4=(0, 0, 0, \widetilde{d}_4), \quad (p_4, q_4, r_4, s_4)U_4=(0, 0, 0, \widetilde{s}_4),$$ and $f(z)$ and $g(z)$ are unitarily equivalent to 
		$$\left(\begin{array}{ccccccc}
		1,&0,& a_1z+a_2z^2+a_3z^3, & b_1z+b_2z^2+b_3z^3, & c_1z+c_2z^2+c_3z^3, &  \sum_{i=1}^{4}d_iz^i\\
		0,& 1,& p_1z+p_2z^2+p_3z^3, & q_1z+q_2z^2+q_3z^3, & r_1z+r_2z^2+r_3z^3, & \sum_{i=1}^{4}s_iz^i \end{array}\right),$$
		where $I_2$ is the $2 \times 2$ unit matrix and $U_4  \in U(4)$ (in the absence of confusion, we also use letters $a_i, b_i, c_i$ and $d_i$).
		
		From the fact that $f(z)$ is a holomorphic section of $\phi_{-1}$, it is easy to see that $\overline{f}(z)$ is an antiholomorphic  section of $\underline{\phi}_{1}$, 
		and then we arrive at the following equation 
		$$\langle\overline{f}(z), f(z)\rangle =0$$
		from the fact that $\phi_{-1}$ and $\phi_{1}$ are mutually orthogonal,
		which verifies
		$$d_4=0.$$
		With a similar discussion for $g(z)$ we also obtain $s_4=0$ and then, $f(z)$ and $g(z)$ become
		$$\left(\begin{array}{ccccccc}
		1,&0,& a_1z+a_2z^2+a_3z^3, & b_1z+b_2z^2+b_3z^3, & c_1z+c_2z^2+c_3z^3, & \sum_{i=1}^{3}d_iz^i\\
		0,& 1,& p_1z+p_2z^2+p_3z^3, & q_1z+q_2z^2+q_3z^3, & r_1z+r_2z^2+r_3z^3, & \sum_{i=1}^{3}s_iz^i\end{array}\right).$$

		Using the same method, it is not difficult for us to get
		$a_3=b_3=c_3=d_3=0$ and $p_3=q_3=r_3=s_3=0$   and $f(z)$ and $g(z)$ can be finally expressed as $$\left(\begin{array}{ccccccc}
		1,&0,& a_1z+a_2z^2, & b_1z+b_2z^2, & c_1z+c_2z^2, & d_1z+d_2z^2\\
		0,& 1,& p_1z+p_2z^2, & q_1z+q_2z^2, & r_1z+r_2z^2, & s_1z+s_2z^2 \end{array}\right).$$
		Therefore for any holomorphic section of $\phi_{-1}$, its degree $\leq 2$. This together with  \eqref{eq:5.3} implies that 
		$$m=2,$$  which finishes the proof.
	\end{proof}

	With Proposition 5.3,  the harmonic sequence given in  \eqref{eq:5.3} becomes
	\begin{equation}0\stackrel{\partial{''}} {\longleftarrow}
		\overline{\underline{f}}^{(2)}_2\stackrel{\partial{''}} {\longleftarrow}
		\overline{\underline{f}}^{(2)}_{1}\stackrel{\partial{''}} {\longleftarrow}
		\underline{\varphi}=\underline{\overline{W}}\oplus\underline{W}\stackrel {\partial{'}}
		{\longrightarrow}\underline{f}^{(2)}_{1} \stackrel {\partial{'}}
		{\longrightarrow}\underline{f}^{(2)}_2 \stackrel{\partial{'}}
		{\longrightarrow}0. \label{eq:5.9}\end{equation}
	Here $\underline{f}^{(2)}_{0}$ is a subbundle with rank one of $\underline{\varphi}$, let $\alpha = f^{(2)\bot}_{0} \cap\varphi$, then it satisfies
	$\underline{\alpha}= \underline{ker}A'_{\varphi}$ and $\underline{f}^{(2)}_{0}=\underline{\alpha}^{\bot}\cap\underline{\varphi}=\underline{ker}A^{'\perp}_{\varphi}$,
	which establishes that
	\begin{equation}A''_{\varphi}(\underline{ker}A^{'\perp}_{\varphi})= 0.\label{eq:5.10}\end{equation}
	In \eqref{eq:5.9}, $\overline{\underline{f}}^{(2)}_{0}$ and $\underline{f}^{(2)}_{0}$ are both  subbundles of $\underline{\varphi}$. \eqref{eq:5.10} together with relation $\underline{\varphi}=\underline{\alpha}\oplus \underline{f}^{(2)}_{0}=\underline{\overline{W}}\oplus\underline{W}$ imply that
	$$\underline{\alpha}=\underline{W}=\overline{\underline{f}}^{(2)}_{0}, \quad \underline{\overline{W}}=\underline{f}^{(2)}_{0},$$ i.e.
	$\underline{\varphi}=\overline{\underline{f}}^{(2)}_{0}\oplus\underline{f}^{(2)}_{0}$ is a real mixed pair, which is consist with (iii) of Lemma 5.2.
	Harmonic sequences \eqref{eq:5.2} and \eqref{eq:5.9} become
	$$0\stackrel{\partial{''}} {\longleftarrow}
	\overline{\underline{f}}^{(2)}_2\stackrel{\partial{''}}
	{\longleftarrow}\underline{f}^{(2)}_0\oplus\underline{\overline{f}}^{(2)}_{1}
	\stackrel{\partial{''}} {\longleftarrow}
	\widetilde{\underline{\phi}}= \overline{\underline{f}}^{(2)}_0\oplus\underline{f}^{(2)}_{1}\stackrel {\partial{'}}
	{\longrightarrow}\underline{f}^{(2)}_2\stackrel{\partial{'}}
	{\longrightarrow}0,$$
	$$0\stackrel{\partial{''}} {\longleftarrow}
	\overline{\underline{f}}^{(2)}_2\stackrel{\partial{''}} {\longleftarrow}
	\overline{\underline{f}}^{(2)}_{1}\stackrel{\partial{''}} {\longleftarrow}
	\underline{\varphi}=\underline{\overline{f}}^{(2)}_0\oplus\underline{f}^{(2)}_0\stackrel {\partial{'}}
	{\longrightarrow}\underline{f}^{(2)}_{1} \stackrel {\partial{'}}
	{\longrightarrow}\underline{f}^{(2)}_2 \stackrel{\partial{'}}
	{\longrightarrow}0.$$ Here $\underline{\overline{f}}^{(2)}_{2},  \underline{\overline{f}}^{(2)}_{1}, \underline{\overline{f}}^{(2)}_{0},  \underline{f}^{(2)}_{0},  \underline{f}^{(2)}_{1},  \underline{f}^{(2)}_{2}$ are mutually orthogonal, and $X$ can be put as 
	$$X= \overline{f}^{(2)}_1+x_1f^{(2)}_0,$$ 
	where  $x_1$ is a smooth function on $S^2$ expect some isolated
	points. Let $$Y= -|f^{(2)}_0|^2x_1f^{(2)}_1+|f^{(2)}_1|^2\overline{f}^{(2)}_0,$$ it satisfies $\underline{Y}=\underline{\overline{X}}^{\bot}\cap \widetilde{\underline{\phi}}$. Applying the equation $\underline{Y}=A'_{\phi}|\underline{X}$ we obtaion
	\begin{equation}\partial{x_1}+x_1\partial{\log{|f^{(2)}_0|^2}}=0.\label{eq:5.11}\end{equation} Hence we have

	\begin{prop}
		Let   $\phi: S^2\rightarrow G(2,6;\mathbb{R})$  be a linearly  
		full irreducible totally unramified  (strongly) isotropic harmonic map with constant curvature, then $\underline{\phi}=
		\underline{\overline{X}} \oplus \underline{X}$ with $X= \overline{f}^{(2)}_1+x_1f^{(2)}_0$, where $\overline{f}^{(2)}_2, \overline{f}^{(2)}_1, \overline{f}^{(2)}_0, f^{(2)}_0, f^{(2)}_1, f^{(2)}_2$ are mutually orthogonal and the corresponding coefficient $ x_1$ satisfies  equation \eqref{eq:5.11}.
	\end{prop}

	For any linearly full
	irreducible totally unramified harmonic map $\phi:S^2 \rightarrow G(2,6;\mathbb{R})$ with constant curvature $K$ and isotropy order $r=\infty$,  from the above discussion we easily see that
	$$\Omega_0 = \left(\begin{array}{ccccccc}
	\frac{\langle
		\partial \overline{X} , Y\rangle }{|X||Y|} & -\frac{|f^{(2)}_1|}{|f^{(2)}_0|} \\
	\frac{|f^{(2)}_2|}{|X|} & 0 \end{array}\right),   \quad \Omega_{-1} =
	\left(\begin{array}{ccccccc}\frac{|f^{(2)}_1|}{|f^{(2)}_0|} & 0 \\
	-\frac{\langle
		\partial \overline{X} , Y\rangle }{|X||Y|} & -\frac{|f^{(2)}_2|}{|X|}
	\end{array}\right),$$
	\begin{equation}L_0 =L_{-1}= \frac { \langle
			\partial \overline{X} , Y\rangle \langle Y, \partial \overline{X} \rangle }{|X|^2|Y|^2} +
		\frac{|f^{(2)}_2|^2}{|X|^2} + l^{(2)}_0,   \  L_1=0, \label{eq:5.12}\end{equation}
	\begin{equation}|\det \Omega_0|^2 dz^2 d{\overline{z}}^2 = [l^{(2)}_0]^2l^{(2)}_1
		\frac{|f^{(2)}_0|^2}{|X|^2}dz^2
		d{\overline{z}}^2,\label{eq:5.13}\end{equation}
	and 
	\begin{equation}\partial \overline{\partial} \log|\det \Omega_0|^2 =L_{-1} - 2
		L_{0} + L_{1}\label{eq:5.14}\end{equation}  by direct computation.

	From the assumption that $\phi$ is totally unramified,  we find $|\det
	\Omega_0|^2 dz^2 d{\overline{z}}^2 \neq 0$ everywhere on $S^2$ and $\underline{f}_0^{(2)}, \ \underline{f}_1^{(2)}, \
	\underline{f}_2^{(2)}: S^2 \rightarrow \mathbb{C}P^{2} \subset
	\mathbb{C}P^{6}$ are also  totally unramified from \eqref{eq:5.13}. In this case, we prove

	\begin{prop}
		Let $\phi:S^2 \rightarrow G(2,6;\mathbb{R})$ be a linearly full
		irreducible totally  unramified harmonic map of Gauss curvature
		$K$.  Suppose that $K$ is constant, then $\phi$ is  (strongly) isotropic and totally geodesic and, up to an isometry of $G(2,6;\mathbb{R})$, $
		\underline{\phi} =
		\overline{\underline{U}}\overline{\underline{V}}^{(2)}_1 \oplus
		\underline{U}\underline{V}^{(2)}_1 $ with $K = \frac{1}{2}$ for some $U\in G_W$, where W has the form \eqref{eq:4.17}. Here  W can be found in different types, thus exist different  $U \in U(6)$ such that $UV^{(2)}_1$ are linearly full in $G(2,6;\mathbb{R})$, and they are not $SO(6)-$ equivalent.
	\end{prop}
	\begin{proof}
		Consider local lift of the $i$-th osculating curve $F_i^{(2)} = f^{(2)}_0
		\wedge \ldots \wedge f^{(2)}_i$ ($i=0, 1, 2$), here we choose a nowhere
		zero holomorphic $\mathbb{C}^6$-valued function $f^{(2)}_0$ such that $F_i^{(2)}$ is
		a nowhere zero holomorphic curve and it is a polynomial function on
		$\mathbb{C}$ of degree $\delta^{(2)}_i$ satisfying
		$\partial\overline{\partial}\log |F_i^{(2)}|^2 = l^{(2)}_i$. So using
		\eqref{eq:5.7} \eqref{eq:5.12}  \eqref{eq:5.13}
		and \eqref{eq:5.14}, we obtain
		\begin{equation}
			\partial \overline{\partial} \log \frac{ (1 +
				z\overline{z})^{4}|f^{(2)}_0|^2}{|F_0^{(2)}|^6|X|^2}
			=0.\label{eq:5.15}\end{equation}
		From \eqref{eq:5.13} we know that
		$\frac{|f^{(2)}_0|^2}{|X|^2} l^{(2)}_0$ is a globally defined function without
		zeros on $S^2$. Then it follows from \eqref{eq:3.3}  that $$\frac{ (1 +
			z\overline{z})^{4}|f^{(2)}_0|^2}{|F_0^{(2)}|^6|X|^2}=\frac{(1+z\overline{z})^{4}}{|F_0^{(2)}|^2|F_1^{(2)}|^2}\cdot \frac{|f^{(2)}_0|^2}{|X|^2}l_0^{(2)}$$ is globally defined
		on $\mathbb{C}$ and has a positive constant limit $\frac{1}{c}$ as
		$z\rightarrow \infty$. Thus  \eqref{eq:5.15} gives us that
		$$\frac{ (1 +
			z\overline{z})^{4}|f^{(2)}_0|^2}{|F^{(2)}_0|^6|X|^2} = \frac{1}{c},$$
		i.e.
		\begin{equation}|X|^2 = \frac{c(1+z\overline{z})^{4}}{|f^{(2)}_0|^4}.\label{eq:5.16}\end{equation}
		Applying the equation $X  = \overline{f}^{(2)}_1 +  x_3f^{(2)}_0$,
		\eqref{eq:5.16} can be rearranged as
		\begin{equation}|x_3|^2|F^{(2)}_0|^4 + |F^{(2)}_1|^2 =
			\frac{c(1+z\overline{z})^{4}}{|f^{(2)}_0|^2}.
			\label{eq:5.17}\end{equation}

		In view of \eqref{eq:5.11} we get $\partial(x_3|f^{(2)}_0|^2) =
		0$. 
		Observing \eqref{eq:5.17}, from the fact that  both $|F_1^{(2)}|^2$
		and $\frac{(1+z\overline{z})^{4}}{|f^{(2)}_0|^2}$ have no singular
		points except $z=\infty$,  we have $x_3|f^{(2)}_0|^2$ is a antiholomorphic
		function on $\mathbb{C}$ at most with the pole $z=\infty$. So it is
		a polynomial function about $\overline{z}$. Without loss of generality,  set
		$$x_3|f^{(2)}_0|^2 = h(\overline{z}),$$ the formula \eqref{eq:5.17} is rewritten as \begin{equation}|h(\overline{z})|^2+|F^{(2)}_1|^2 =
			\frac{c(1+z\overline{z}) ^{4}}{|f^{(2)}_0|^2}.
			\label{eq:5.18}\end{equation} Since both sides of \eqref{eq:5.18}
		are polynomial functions and $\delta^{(2)}_0=2$, then we have
		\begin{equation}|f^{(2)}_0|^2 = \mu (1+z\overline{z})^2, \label{eq:5.19}\end{equation} where $\mu$ is a real
		parameter.

		Here we claim that $h=0$. Otherwise if $h\neq 0$, then $1+z\overline{z}$ is a factor of it, which
		contracts the fact that $h$ is antiholomorphic. Thus we have $h=0$,
		which implies that the function $x_1$ should vanish, i.e. $x_1=0$. Then
		$$X = \overline{f}_1^{(2)} , \quad \underline{\phi} = \underline{\overline{f}}_1^{(2)} \oplus
		\underline{f}_1^{(2)}.$$

		As to the second fundamental form $B$ of $\phi$,  by \eqref{eq:2.7} and a series of calculations, we obtain
		$$\left\{
		\begin{array} {l}
		\partial\phi=\frac{1}{|f_1^{(2)}|^2}[\overline{f}_1^{(2)}(\overline{f}_{2}^{(2)})^{\ast}+f_{2}^{(2)}f_1^{(2)\ast}]-\frac{1}{|f^{(2)}_0|^2}[\overline{f}^{(2)}_0(\overline{f}^{(2)}_{1})^{\ast}+f^{(2)}_{1}f_0^{(2)\ast}],
		\\ A_z=\frac{1}{|f^{(2)}_1|^2}[\overline{f}_1^{(2)}(\overline{f}_{2}^{(2)})^{\ast}-f_{2}^{(2)}f_1^{(2)\ast}]+\frac{1}{|f^{(2)}_0|^2}[\overline{f}_0^{(2)}(\overline{f}^{(2)}_{1})^{\ast}-f^{(2)}_{1}f_0^{(2)\ast}],
		\\ P=0.
		\end{array} \right.$$  It is trival that $\|B\|^2=0$, i.e. $\phi$
		is totally geodesic.

		From \eqref{eq:5.19}, by Lemma 3.4, up to a holomorphic isometry of
		$\mathbb{C}P^{5}$, ${f}_1^{(2)}$ is a Veronese surface. We can choose a
		complex coordinate $z$ on $\mathbb{C}=S^2\backslash\{pt\}$ so that $
		{f}_1^{(2)} = {U}{V}^{(2)}_1$, where $U\in U(6)$ and ${V}^{(2)}_1$ has the
		standard expression given in Section 3 (adding zeros to
		${V}^{(2)}_1$ such that ${V}^{(2)}_1 \in \mathbb{C}^6$). Thus we
		have
		$$\underline{\phi}=\underline{\overline{U}}\underline{\overline{V}}^{(2)}_1\oplus
		\underline{U}\underline{V}^{(2)}_1.$$   
		This finishes the proof.
	\end{proof}

	\begin{remark}
		\emph{In Proposition 5.5, to determine} $\phi$\emph{, we
			just need to determine the matrix} $U$.
		\emph{From the above discussion, we have}
		$$\langle {f}_0^{(2)}, \overline{f}_2^{(2)}\rangle=0,$$ \emph{which is equivalent to}
		\begin{equation}\textrm{tr} U^{T}UV^{(2)}_0 V^{(2)T}_{2}= 0, \label{eq:5.20}
		\end{equation}  \emph{here} $V^{(2)}_0V^{(2)T}_2$ \emph{is a
			polynomial matrix in} $z$ \emph{and} $\overline{z}$ \emph{by the standard expressions of} $V^{(2)}_0$ \emph{and} $V^{(2)}_2$, \emph{and} $U$  \emph{is a constant matrix.
			Using the method of indeterminate coefficients by}
		\eqref{eq:5.20}, \emph{put} $U^TU=(w_{ij}), \ 0\leq i,j \leq 5$,
		\emph{by direct computation it can be expressed in the same with  \eqref{eq:4.17}.} \emph{Here an easy example is to set} $U=U_0$ \emph{as the one shown in} \eqref{eq:4.18}.  \emph{In this case} $\underline{\phi} = \overline{\underline{U}}_0\overline{\underline{V}}^{(2)}_1 \oplus
		\underline{U}_0\underline{V}^{(2)}_1 =
		\overline{\underline{f}}^{(2)}_1 \oplus
		\underline{f}^{(2)}_1 $ \emph{has Gauss curvature} $K =
		\frac{1}{2}$, \emph{where} \begin{equation}f^{(2)}_1 = U_0V^{(2)}_1=[(-\sqrt{2}\overline{z}, -\sqrt{-2}\overline{z}, 1-z\overline{z}, \sqrt{-1}(1-z\overline{z}), \sqrt{2}z, \sqrt{-2}z)^{T}]. \label{eq:5.21}\end{equation}
	\end{remark}

	By Proposition 4.3 and Proposition 5.5, we conclude a
	classification of conformal minimal immersions of constant curvature
	from $S^2$ to $G(2,6;\mathbb{R})$ as follows:

	\begin{theorem}
		Let $\phi:S^2 \rightarrow G(2,6;\mathbb{R})$ be a linearly full
		conformal minimal immersion with Gauss curvature $K$. Suppose that $K$ is constant, then, up to an isometry of $G(2,6;\mathbb{R})$,
		\\(i) If $\phi$ is reducible with finite isotropy order,   either $\underline{\phi}=\overline{\underline{U}}\overline{\underline{V}}^{(3)}_0\oplus \underline{U}\underline{V}^{(3)}_0$ with $K= \frac{2}{3}$, or $\underline{\phi}=\overline{\underline{U}}\overline{\underline{V}}^{(2)}_0\oplus \underline{U}\underline{V}^{(2)}_0$ with $K= 1$  for some $U\in G_W$, where W has the form \eqref{eq:4.13} or \eqref{eq:4.15}  respectively;
		\\(ii) If $\phi$ is reducible and (strongly) isotropic,   either $\underline{\phi} = \underline{U}\underline{V}^{(4)}_2 \oplus
		\underline{c}_0 $ with  $K= \frac{1}{3}$ for some $U \in U(5)$, or $\underline{\phi}=\overline{\underline{U}}\overline{\underline{V}}^{(2)}_0\oplus \underline{U}\underline{V}^{(2)}_0$ with $K= 1$ for some $U\in G_W$, where W has the form \eqref{eq:4.17};
		\\(iii) If $\phi$ is totally unramified  irreducible, then, it is (strongly) isotropic and totally geodesic,  $ \underline{\phi} =
		\overline{\underline{U}}\overline{\underline{V}}^{(2)}_1 \oplus
		\underline{U}\underline{V}^{(2)}_1 $ with  $\ K =
		\frac{1}{2}$ for some $U\in G_W$, where W has the form \eqref{eq:4.17}.
		In each case, there are many different types of W, thus exist different  $U \in U(6)$ such that corresponding $UV^{(m)}_0 (m=2, 3)$ or $UV^{(2)}_1$ are  linearly full in $G(2,6;\mathbb{R})$, and they are not $SO(6)-$ equivalent.
	\end{theorem}

	Theorem 5.7 shows that, up to an isometry of
	$G(2,6;\mathbb{R})$,  conformal minimal immersions of constant
	curvature from $S^2$ to $G(2,6;\mathbb{R})$, or equivalently, a
	complex hyperquadric $Q_{4}$  can be presented by the Veronese
	surfaces in $\mathbb{C}P^{2},      \mathbb{C}P^{3}$ or $\mathbb{C}P^{4}$.

	Let $\phi$ be a linearly full conformal minimal immersion
	of constant curvature from $S^2$ to $G(2,6;\mathbb{R})$,  and
	$f_{\phi}$ be the corresponding map of $\phi$ from $S^2$  to
	$Q_{4}$. In Theorem 5.7,   for cases  $\underline{\phi}=\overline{\underline{U}}\overline{\underline{V}}^{(2)}_0\oplus \underline{U}\underline{V}^{(2)}_0$,  $\underline{\phi}=\overline{\underline{U}}\overline{\underline{V}}^{(3)}_0\oplus \underline{U}\underline{V}^{(3)}_0$  and $\underline{\phi} =
	\overline{\underline{U}}\overline{\underline{V}}^{(2)}_1 \oplus
	\underline{U}\underline{V}^{(2)}_1$,  the corresponding maps of $\phi$ from  $S^2$ to $Q_4$ are minimal (cf.  \eqref{eq:4.14},  \eqref{eq:4.16}, \eqref{eq:4.20}, \eqref{eq:5.21}), which are also minimal from  $S^2$ to $\mathbb{C}P^5$. But for $\underline{\phi} = \underline{U}_1\underline{V}^{(4)}_2 \oplus
	\underline{c}_0 $, from \eqref{eq:4.21},  the corresponding map $f_{\phi}: S^2\rightarrow Q_4$ of $\phi$ is as follows
	$$\begin{array}{lll} \underline{f}_{\phi} & = & \left[ (\sqrt{3}(z^2 + \overline{z}^2), \
	\sqrt{-3}( \overline{z}^2 - z^2 ), \
	\sqrt{3}(z+\overline{z})(z\overline{z} - 1), \right. \\
	& &\left.  \sqrt{-3}(\overline{z} - z)(z\overline{z} - 1), \
	1-4z\overline{z} + z^2\overline{z}^2, \ \sqrt{-1}(1+z\overline{z})^2
	)^T \right]: S^2 \rightarrow Q_4 \subset \mathbb{C}P^5.\end{array}$$
	By a simple test, we can check that it is not minimal
	in $\mathbb{C}P^5$.

\end{document}